\documentclass{amsart}
\usepackage[utf8]{inputenc}

\usepackage{hyperref}
\usepackage{amsmath}
\usepackage{amsmath,amsthm,amssymb,amsfonts}
\usepackage[utf8]{inputenc}
\usepackage{enumerate}
\usepackage{graphicx}
\usepackage{enumitem}
\usepackage{tikz}
\usepackage[numbers]{natbib}
\newcommand{\f}{\overline}
\theoremstyle{plain}
\newtheorem{theorem}{Theorem}[section]
\newtheorem{lemma}[theorem]{Lemma}
\newtheorem{claim}[theorem]{Claim}
\newtheorem{corollary}[theorem]{Corollary}
\newtheorem{proposition}[theorem]{Proposition}

\newtheorem{question}[theorem]{Question}

\newtheorem{fact}[theorem]{Fact}
\newcommand{\thistheoremname}{}
\newtheorem*{genericthm*}{\thistheoremname}
\newenvironment{namedthm*}[1]
{\renewcommand{\thistheoremname}{#1}%
	\begin{genericthm*}}
	{\end{genericthm*}}

\theoremstyle{definition}\newtheorem{definition}[theorem]{Definition}

\theoremstyle{definition}
\theoremstyle{definition}\newtheorem{remark}[theorem]{Remark}
\theoremstyle{definition}\newtheorem{notation}[theorem]{Notation}

\newcommand{\mc}{\mathcal}

\newcommand{\Q}{\mathbb{Q}}

\newcommand{\N}{\mathbb{N}}
\newcommand{\Z}{\mathbb{Z}}

\DeclareMathOperator{\aut}{Aut}
\DeclareMathOperator{\id}{id}

\DeclareMathOperator{\dom}{dom}
\DeclareMathOperator{\ran}{ran}

\DeclareMathOperator{\rd}{rd}

\newcommand{\nr}{\lnot R}

\begin{document}
	\title[The structure of random automorphisms of $\mc{R}$]{The structure of random automorphisms of the random graph}

		\author[U. B. Darji]{Udayan B. Darji}
		\address{Department of Mathematics, University of Louisville,
			Louisville, KY 40292, USA\\Ashoka University, Rajiv Gandhi Education City, Kundli, Rai 131029, India} 
		\email{ubdarj01@louisville.edu}
		\urladdr{http://www.math.louisville.edu/\!$\tilde{}$ \!\!darji}
		\author[M. Elekes]{M\'arton Elekes}
		\address{Alfr\'ed R\'enyi Institute of Mathematics, Hungarian Academy of Sciences,
			PO Box 127, 1364 Budapest, Hungary and E\"otv\"os Lor\'and
			University, Institute of Mathematics, P\'azm\'any P\'eter s. 1/c,
			1117 Budapest, Hungary}
		\email{elekes.marton@renyi.mta.hu}
		\urladdr{www.renyi.hu/ \!$\tilde{}$ \!\!emarci}
		
		\author[K. Kalina]{Kende Kalina}
		\address{E\"otv\"os Lor\'and
			University, Institute of Mathematics, P\'azm\'any P\'eter s. 1/c,
			1117 Budapest, Hungary}
		\email{kkalina@cs.elte.hu}
		\author[V. Kiss]{Viktor Kiss}
		\address{Alfr\'ed R\'enyi Institute of Mathematics, Hungarian Academy of Sciences,
			PO Box 127, 1364 Budapest, Hungary and E\"otv\"os Lor\'and
			University, Institute of Mathematics, P\'azm\'any P\'eter s. 1/c,
			1117 Budapest, Hungary}
		\email{kiss.viktor@renyi.mta.hu}
		
		\author[Z. Vidny\'anszky]{Zolt\'an Vidny\'anszky}
		\address{Kurt G\"{o}del Research Center for Mathematical Logic, 
			Universit\"{a}t Wien,
			W\"{a}hringer Stra{\ss}e 25, 
			1090 Wien,
			Austria and Alfr\'ed R\'enyi Institute of Mathematics, Hungarian Academy of Sciences,
			PO Box 127, 1364 Budapest, Hungary}
		\email{zoltan.vidnyanszky@univie.ac.at}
		\urladdr{
			http://www.logic.univie.ac.at/~vidnyanszz77/
		}
		
		\thanks{The second, fourth and fifth authors were partially supported by the
			National Research, Development and Innovation Office
			-- NKFIH, grants no.~113047, no.~104178 and no. ~124749. The fifth
			author was also supported by FWF Grant P29999.}
	\begin{abstract}
	We give a complete description of the size of the conjugacy classes of the automorphism group of the random graph with respect to Christensen's Haar null ideal. It is shown that every non-Haar null class contains a translated copy of a nonempty portion of every compact set and that there are continuum many non-Haar null conjugacy classes. Our methods also yield a new proof of an old result of Truss.

	\end{abstract}
	\maketitle

	The investigation of the size of the conjugacy classes of topological groups with respect to the meager ideal is an important field with several applications (see e. g. \cite{bernardes},\cite{glasner2003universal},\cite{shelah},\cite{KechrisRosendal},\cite{truss1992generic}). Here, following Dougherty and Mycielski \cite{DM}, we take a different, measure-theoretic perspective (for a detailed introduction we refer the reader to \cite{autgencikk}). Let us consider the following notion of smallness:
	
	\begin{definition}[Christensen, \cite{originalhaarnull}]
		\label{d:haarnull}
		Let $G$ be a 
		Polish group and $B \subset G$ be Borel. We say that $B$ is \textit{Haar null} if 
		there exists a
		Borel probability measure $\mu$ on $G$ such that for every $g,h \in G$
		we have $\mu(gBh)=0$. 
		An arbitrary set $S$ is called Haar null if $S \subset B$ for some Borel Haar 
		null set $B$.
	\end{definition}
	
	Using this definition, it makes sense to talk about the properties of random 
	elements of a Polish group. A property $P$ of elements of a Polish group $G$ is said 
	to \textit{hold almost surely} or \emph{almost every element of G has property 
		$P$} if the set
	$\{g \in G: g \text{ has property } P\}$ is co-Haar null. 
	
	In \cite{autgencikk} we have generalized the results of Dougherty and Mycielski to a large class of automorphism groups of countable structures as follows.
	
	\begin{definition} \label{NACdef} Let $G$ be a closed subgroup of $S_\infty$. We say that \textit{$G$ has the finite algebraic closure property ($FACP$)} if for every finite $S \subset \omega$ the set $\{b:|G_{(S)}(b)|<\infty\}$ is finite, where $G_{(S)}$ denotes the pointwise stabilizer of the set $S$.
	\end{definition}
	
	Among other things, the following theorem has been proved. 
	
	\begin{theorem}
		\label{t:gen}
		Let $G<S_\infty$ be closed. Then the following are equivalent:
		\begin{enumerate}
			\item almost every element of $G$ has finitely many finite orbits,
			\item $G$ has the FACP.
		\end{enumerate}
		Moreover, any of the above conditions implies that almost every element of $G$ has infinitely many infinite orbits.
	\end{theorem}

	Unfortunately, the above theorem is typically far from being a complete description of the size of the conjugacy classes of an automorphism group of a given countable structure. The aim of the current paper is to solve this question in a special case, namely, to give a complete description of the size of the conjugacy classes of the automorphism group of the random graph $\mc{R}=(V,R)$, that is, the unique countable graph, having the following property: for every pair of finite disjoint sets $A,B \subset V$ there exists $v \in V$ such that $(\forall x \in A)(x R v)$ and $(\forall y \in B)(y \nr v)$. Note that $\aut(\mc{R})$ has the FACP. 
	
	For $f \in \aut(\mc{R})$ and $A \subset V$ let us use the notation $\mc{O}^f(A)$ for the set $\{f^k(v):v \in A, k \in \Z\}$. Our characterization reads as follows. 
	
	\begin{namedthm*}{Theorem \ref{t:randomintro}} For almost every element $f$ of $\aut(\mc{R})$
		\begin{enumerate}
			\item for every pair of finite disjoint sets, $A,B \subset V$ there exists $v \in V$ such that $(\forall x \in A)( x R v)$ and $(\forall y \in B)(y \nr v)$  \textit{ and $v \not \in \mathcal{O}^f(A \cup B)$, i. e., the union of orbits of the elements of $A \cup B$},
			\item (from Theorem \ref{t:gen}) $f$ has only finitely many finite orbits. 
		\end{enumerate}
		These properties characterize the non-Haar null conjugacy classes, i. e., a conjugacy class is non-Haar null if and only if one (or equivalently each) of its elements has properties \eqref{pt:autr1} and \eqref{pt:autr2}.
		
		Moreover, every non-Haar null conjugacy class contains a translate of a portion\footnote{If $K$ is a compact set and U is an open set, a nonempty set of the form $U \cap K$ will be called a \emph{portion of $K$}.} of every compact set and those non-Haar null classes in which the elements have no finite orbits contain a translate of every compact set. 
	\end{namedthm*}

	For a given function $p:\N \setminus \{0\} \to 2$ one can construct inductively an $f_p \in Aut(\mc{R})$ such that $f_p$ has properties $(1)$ and $(2)$ from the above theorem and for every $v \in V$ and $n \in \N \setminus \{0\}$ we have $vRf^n(v) \iff p(n)=0$. Since it is easy to see that for $p \not =p'$ automorphisms of the form $f_p$ and $f_{p'}$ cannot be conjugate we obtain the following corollary:
	\begin{corollary}
		\label{c:autrcont}
		There are continuum many non-Haar null classes in $\aut(\mc{R})$ and their union is co-Haar null.
	\end{corollary}
	
	Note that it was proved by Solecki \cite{openlyhaarnull} that in every non-locally compact Polish group that admits a two-sided invariant metric there are continuum many pairwise disjoint non-Haar null Borel sets, thus the above corollary is an extension of his results for $\aut(\mc{R})$.  
	
	In the proof we use a version (see Lemma \ref{l:split}) of the following lemma  which is interesting in itself.
	
	\begin{lemma} (Splitting Lemma, finite version) If $F \subset \aut(\mc{R})$ is a finite set and $A,B \subset V$ are disjoint finite sets, then there exists a vertex $v$ so that for every distinct $f,g \in F$ we have $f(v) \not =g(v)$, $(\forall x \in A)( x R v)$ and $(\forall y \in B)(y \nr v)$.
	\end{lemma}

	From the above theorem and the Splitting Lemma one can give a new proof of well known results of Truss \cite{truss1985group} (which was improved by him later) and Rubin, that states that if $f,g $ are non-identity elements in $\aut(\mc{R})$ then $g$ is the product of four conjugates of $f$, see Theorem \ref{t:truss}.
	
	Finally, we would like to point out that a similar characterization result can be proved for $\aut(\Q)$, the automorphism group of the rational numbers (as an ordered set) see \cite{autqcikk} and \cite{autgencikk}. Interestingly, the proof is completely different, hence the following question is very natural:
	\begin{question}
		Is it possible to unify these proofs? Are there necessary and sufficient model theoretic conditions which characterize the measure theoretic behavior of the conjugacy classes?
	\end{question}
		
	The paper is organized as follows. First, in Section \ref{s:prel} we summarize facts and notations used later, then in Section \ref{s:autr} we prove our main theorem. We present an application of our theorem in Section \ref{s:appl}.

	\section{Preliminaries and notations}
	\label{s:prel}
	We will follow the notations of \cite{kechrisbook}. For a detailed introduction to the theory of Polish groups see \cite[Chapter 1]{becker1996descriptive}, while the model theoretic background can be found in \cite[Chapter 7]{hodges}. Nevertheless, we summarize the basic facts which we will use.

	As mentioned before, $S_\infty$ stands for the permutation group of the countably infinite set $\omega$. It is well known that $S_\infty$ is a Polish group with the pointwise convergence topology. This coincides with the topology generated by the sets of the form $[p]=\{f \in S_\infty: p \subset f\}$, where $p$ is a finite partial permutation.
	
	Let $\mathcal{A}$ be a countable structure. By the countability of $\mc{A}$, every automorphism $f \in \aut(\mc{A})$ can be regarded as an element of $S_\infty$, and it is not hard to see that in fact $\aut(\mc{A})$ will be a closed subgroup of $S_\infty$. Moreover, the converse is also true, namely every closed subgroup of $S_\infty$ is isomorphic to the automorphism group of a countable structure. 
	
	\begin{notation}
		\label{n:impo}	
	We fix an enumeration of $\{v_0,v_1,\dots\}$ of $V$, the vertex set of the random graph. If $\mc{K} \subset \aut(\mathcal{R})$ and $M \subset V$ then $\mc{K}(M)=\{f(v):v \in M,f \in \mc{K}\}$, similarly $\mc{K}^{-1}(M)=\{f^{-1}(v):v \in M, f \in \mc{K}\}$ and $\mc{K}|_M=\{f|_M:f \in \mc{K}\}$. For a set $M \subset V$ we will denote by $M^*$ the set $M \cup \mc{K}^{-1}(M)$. We shall also abuse this notation, for $v \in V$ letting $\mc{K}(v)=\mc{K}(\{v\})$. Moreover, we will also use the notation $\mc{K}^2=\{ff':f,f' \in \mc{K}\}$ and $\mc{K}^{-1}=\{f^{-1}:f \in \mc{K}\}$. If $f$ is a function let us use the notation $\rd(f)$ for the set $\ran(f) \cup \dom(f)$. 
	
	\end{notation}
	
	We will constantly use the following fact.
	
	\begin{fact}
		\label{f:compactchar} Let $\mathcal{A}$ be a countable structure. A closed subset $\mathcal{K}$ of $\aut(\mathcal{A})$ is compact if and only if for every $M$ finite set the set $\mc{K}(M) \cup \mc{K}^{-1}(M)$ is finite. In particular, for a compact set $\mc{K}$ the set $\mc{K}|_M$ is also finite. 
	\end{fact}

	Let us consider the following notion of largeness:
	
	\begin{definition}
		\label{d:catcherbiter}
		Let $G$ be a Polish topological group. A set $A \subset G$ is called \textit{compact catcher} if for every compact $K \subset G$ there exist $g,h \in G$ so that $gKh \subset A$.  $A$ is \textit{compact biter} if for every compact $K \subset G$ there exist an open set $U$ and $g,h \in G$ so that $U \cap K \not = \emptyset$, and $g(U \cap K)h \subset A$. 
	\end{definition}
	 The following easy observation is one of the most useful tools to prove that a certain set is not Haar null.
	\begin{fact} (see \cite{autgencikk})
		\label{f:biter}
		If $A$ is compact biter then it is not Haar null.
	\end{fact}

	It is sometimes useful to consider right and left Haar null sets: a Borel set $B$ is \emph{right (resp. left) Haar null} if there exists a Borel probability measure $\mu$ on $G$ such that for every $g \in G$
	we have $\mu(Bg)=0$ (resp. $\mu(gB)=0$). An arbitrary set $S$ is called \emph{right (resp. left) Haar null} if $S \subset B$ for some Borel right (resp. left) Haar null set $B$. The following observation will be used several times.

	\begin{lemma} (see \cite{autgencikk})
		\label{l:conjugacyinvariant}
		Suppose that $B$ is a Borel set that is invariant under conjugacy. Then $B$ is left Haar null iff it is right Haar null iff it is Haar null.
	\end{lemma}
	
	The next fact, which can be found verbatim in \cite{autgencikk}, will be used in our characterization result. 
	
	\begin{proposition}(\cite[Proposition 4.10]{autgencikk})
		\label{p:F and C co-Haar null}
		Let $G \leq S_\infty$ be a closed subgroup. If $G$ has the $FACP$ then the set
		\begin{equation*}
		\begin{split}
		\mathcal{C} = \{g \in G :& \text{ for every finite } F \subset \omega \text{ and }
		x \in \omega \;(\text{if $G_{(F)}(x)$ is infinite} \\ 
		&\text{then it is not covered by finitely many orbits of $g$})\}  
		\end{split}
		\end{equation*}
		is co-Haar null.  
	\end{proposition}

	\section{The characterization result}
	\label{s:autr}
	In this section we prove our main theorem, starting with the proof of the most important tool, the Splitting Lemma.

	\subsection{The Splitting Lemma}

	\begin{definition}
		Suppose that $M \subset V$ is a finite set and $\tau:M \to 2$ a function. We say that a vertex $v \in V$ \textit{realizes $\tau$} if for every $w \in M$ we have \[w R v \iff \tau(v)=1.\]
	\end{definition}
	
	\begin{definition}
		Let $M \subset V$ be a finite set and $\mc{K} \subset \aut(\mc{R})$ be compact. We call a vertex $v$ a \textit{splitting point for $M$ and $\mc{K}$} if for every $h, h' \in \mc{K}$ so that $h|_M \not = h'|_M$ we have $h(v) \not = h'(v)$ and $h^{-1}(v)\not =h'^{-1}(v)$.
	\end{definition}

	\begin{lemma} 
		\label{l:split} (Splitting Lemma)
		Let $\mc{K} \subset \aut(\mathcal{R})$ be a compact set, $M \subset V$ finite, $\tau: M \to 2$ a function and $n\in \omega$. There exists a splitting point for $M$ and $\mc{K}$, $v \in V \setminus \{v_i:i \leq n\}$ that realizes $\tau$.
	\end{lemma}
	We start the proof of the lemma with a slightly modified special case, namely when we would like to find a splitting point for a pair of automorphisms. 
	\begin{lemma}
		\label{l:for2}
		Let $p,p'$ be finite partial automorphisms, $w_0$ a vertex with $p(w_0) \not =p'(w_0)$ and $N \in \omega$. There exist two disjoint finite sets of vertices $A,A' \subset V \setminus\{v_i: i \leq N\}$ with the following property: for a vertex $v$ if for every $w \in A$ we have $wRv$ and for every $w' \in A'$ we have $w'\lnot Rv$ then $h(v) \not = h'(v)$ for each $h \in [p] \cap 
		\mc{K}$ and $h' \in [p'] \cap \mc{K}$.

	\end{lemma}
	\begin{proof}
		Let us use the notation $\mc{L}=[p] \cap \mc{K}$ and $\mc{L}'=[p'] \cap \mc{K}$. Take a vertex $w_1 \not \in \mc{L}(\{v_i:i \leq N\}) \cup \mc{L}'(\{v_i:i \leq N\})$ with $w_1 R p(w_0)$ and $w_1 \lnot R p'(w_0)$, this can be done by the compactness of $\mc{L}$ and $\mc{L}'$. Now let $A=\mc{L}^{-1}(w_1)$ and $A'=\mc{L}'^{-1}(w_1)$, again these sets are finite by compactness. Moreover, if $x \in A$ then $x=h^{-1}(w_1)$ for some $h \in \mc{L}$. Since $p(w_0) = h(w_0)$ and $w_1Rp(w_0)$, we have that $w_1Rh(w_0)$, hence $h^{-1}(w_1)Rw_0$, that is, $xRw_0$. Analogously, $w_0 \lnot R x$ for every $x \in A'$, in particular $A \cap A'= \emptyset$. Notice that $w_1 \not \in  \mc{L}(\{v_i:i \leq N\}) \cup \mc{L}'(\{v_i:i \leq N\})$ is equivalent to $\emptyset=(\mc{L}^{-1}(w_1) \cup \mc{L}'^{-1}(w_1)) \cap \{v_i:i \leq N\}$, thus $(A \cup A') \cap \{v_i:i \leq N\}= \emptyset$.
		
		Finally, we have to check that $A$ and $A'$ have the required property, so take a vertex $v$ with $wRv$ and $w'\lnot R v$ for every $w \in A$ and $w' \in A'$ and two automorphisms $h \in \mc{L}$ and $h' \in \mc{L}'$. Clearly, $h^{-1}(w_1) \in \mc{L}^{-1}(w_1)=A$ and $h'^{-1}(w_1) \in \mc{L}'^{-1}(w_1)=A'$ so $h^{-1}(w_1)Rv$ and $h'^{-1}(w_1) \lnot R v$, consequently, $w_1R h(v)$ and $w_1 \lnot R h'(v)$, in particular $h(v) \not =h'(v)$.
	\end{proof}
	\begin{proof}[Proof of the Splitting Lemma.]
		
		Let $m_0>n$ so that $M \cup \mc{K}(M) \subset \{v_i: i \leq m_0\}$ and $\mc{K}_1=\mc{K} \cup \mc{K}^{-1}$. By the compactness of $\mc{K}$ the set $M \cup \mc{K}(M)$ is finite and $\mc{K}_1$ is compact. List the pairs of distinct finite partial automorphisms in $\mc{K}_1|_{\{v_i:i \leq m_0\}}$ as $\{(p_j,p'_j):j < k\}$. Again, from the compactness of $\mc{K}_1$ it follows that there are only finitely many such  pairs. Using Lemma \ref{l:for2} we can inductively define a sequence $m_0<m_1<\dots<m_k$ of natural numbers and a sequence of disjoint finite sets $A_j,A'_j \subset \{v_{m_j},v_{m_j+1},\dots ,v_{m_{j+1}}\}$ with the property given by the lemma, that is, for every $j < k$ and $h \in [p_j] \cap \mc{K}_1$ and $h' \in [p'_j] \cap \mc{K}_1$ if a vertex $v$ is connected to every vertex in $A_j$ and not connected to every vertex in $A'_j$ then $h(v) \not = h'(v)$. 
		
		Now take a vertex $v  \in V \setminus \{v_i:i \leq m_k\}$ that realizes $\tau$ and $v$ is connected to each vertex in $\cup_{j < k} A_j$ and not connected to every vertex in $\cup_{j < k} A'_j$. Clearly, the selection of the sequence $(m_j)_{j<k}$ and $(A_j,A'_j)_{j<k}$ shows that such a vertex exists. 
		
		Let $h, h' \in \mc{K}$ be arbitrary with $h|_M \not = h'|_M$ and choose a $w_0 \in M$ so that $h(w_0) \not = h'(w_0)$. Then by the definition of $m_0$ clearly $h(w_0) \in \{v_i:i \leq m_0\}$, moreover $h'^{-1}(h(w_0)) \not =h^{-1}(h(w_0))=w_0$. Consequently, $h^{-1}|_{\{v_i:i \leq m_0\}} \not = h'^{-1}|_{\{v_i:i \leq m_0\}}$. By $M \subset  \{v_0,\dots,v_{m_0}\}$ and using $h,h',h^{-1},h'^{-1} \in \mc{K}_1$ we have that there exist indices $i,j <k$ so that $h|_{\{v_i:i \leq m_0\}}=p_i$ and $h'|_{\{v_i:i \leq m_0\}}=p'_i$ and $h ^{-1}|_{\{v_i:i \leq m_0\}}=p_j$ and $h'^{-1}|_{\{v_i:i \leq m_0\}}=p'_j$, consequently, by the choice of $(A_i,A'_i)$ and $(A_j,A'_j)$ we obtain $h(v) \not = h'(v)$ and $h^{-1}(v) \not = h'^{-1}(v)$, which finishes the proof of the theorem. 
	\end{proof}
	
	\subsection{Translation of compact sets, special case}
	
	In this subsection we will prove that certain types of conjugacy classes are compact biters.

	\begin{definition}
		\label{d:star0}
		Let $f \in \aut(\mc{R})$. We say that $f$ has property $(*)_0$ (resp. $(*)_1$) if 
		\begin{itemize}
			\item $f$ has only finitely many finite orbits and infinitely many infinite orbits,
			\item  for every finite set $M \subset V$ and $\tau:M \to 2$ there exists a $v$ that realizes $\tau$, $v \not \in \mc{O}^f(M)$ and $v \lnot Rf(v)$ (resp. $v Rf(v)$).
		\end{itemize}

	\end{definition}

	\begin{theorem}
		\label{t:randommain1}
		Suppose that $f$ has property $(*)_0$ or $(*)_1$ and denote by $N$ the union of finite orbits of $f$. Suppose that $\mc{K} \subset \aut(\mc{R})$ is a compact set so that for every $h \in \mc{K}$ we have $h|_N=f|_N$. Then $\mc{K}$ can be translated into the conjugacy class of $f$. 
		
		In fact, there exist $g,(\phi_h)_{h \in \mc{K}} \in \aut(\mc{R})$ so that $g|_N=\phi_h|_N=id_N$ and for every $h \in \mc{K}$ we have $\phi_h \circ h \circ g=f \circ \phi_h$. 
	\end{theorem}
	
	Clearly, by the symmetry it is enough to show this theorem for automorphisms having property $(*)_0$. 
	
	The idea of the proof is rather simple: we construct $g$ and $(\phi_h)_{h \in \mc{K}}$ inductively from finite approximations, every time extending the approximations of $g$ by splitting points for $\mc{K}$ and certain finite sets, we also select the new points from far enough (see below the definition of $d_{\mc{K}}$). Using this, we will be able to ensure that the requirements on the extensions of the approximations of $\phi_h$ will not interfere. 
	
	In order to prove the theorem we need a couple of definitions.
	\begin{definition}
		Let us define a new graph with the same vertex set as $\mc{R}$ as follows. Let 
		\[xEy \iff (\exists h \in \mc{K})(h(x)=y \text{ or } h^{-1}(x)=y).\]
		We will denote by $d_{\mc{K}}(x,y)$ the length of the shortest path between $x$ and $y$ and let it be equal to $\infty$ if there is no such path. For sets of vertices $M,M'$ let
		\[d_{\mc{K}}(M,M')=\min \{d_{\mc{K}}(x,y):x \in M, y \in M'\}.\]
		We will denote by $d_{\mc{K}}(M,x)$ the number $d_{\mc{K}}(M,\{x\})$.
	\end{definition}
	
	Note that the function $d_{\mc{K}}:V \times V \to \omega \cup \{\infty\}$ is an extended metric.
	
	\begin{corollary}
		\label{c:splitting}
		Suppose that $M$ is a finite set and $\tau:M \to 2$ is a function. There exists a vertex $v$ that is a splitting point for $M$ and $\mc{K}$, realizes $\tau$ and $d_{\mc{K}}(v,M)>3$.
	\end{corollary}
	
	\begin{proof}
		By the compactness of $\mc{K} \cup \mc{K}^{-1}$ the set \[M \cup \bigcup_{h_1,h_2,h_3 \in \mc{K} \cup \mc{K}^{-1}}h_1h_2h_3(M)\] is finite, so we can take an $n \in \omega$ so that it is contained in $\{v_i: i \leq n\}$. By the Splitting Lemma (Lemma \ref{l:split}) there exists a $v$ so that $v \not \in \{v_i: i \leq n\}$ and $v$ realizes $\tau$. Clearly, $d_{\mc{K}}(v,M)>3$ holds as well.
	\end{proof}
	
	\begin{definition}
		Let $g$ be a finite partial automorphism and $w \in V$. Suppose that for every $i \in \mathbb{Z} \setminus \{0\}$ we have $g^i(w) \not = w$. Then we will denote by $e(w,g)$ the vertex $g^{i}(w)$ so that \[i=\max \{j \in \omega:g^{j}(w)\text{ is defined,}\]\[\text{ i. e.,  for every $k$ with $0  \leq k <j$ we have $g^k(w) \in \dom(g)$} \},\] and similarly we denote by $b(w,g)$ the vertex $g^{-i}(w)$ so that \[i=\max \{j \in \omega:g^{-j}(w)\text{ is defined,}\]\[\text{ i. e.,  for every $k$ with $0  \leq k <j$ we have $g^{-k}(w) \in \ran(g)$} \},\]
		or equivalently, the vertex $e(w,g^{-1})$.
	\end{definition}
	Note that if $w \not \in \dom(g)$ then $e(w,g)=w$ and also if $w \not \in \ran(g)$ then $b(w,g)=w$.

	In the next two definitions we will describe possible set-ups that could be obstacles to carry out the inductive procedure.
	
	\begin{definition}
		\label{d:badsitu}
		Let $h,h' \in \mathcal{K}$ and $g$, $\phi_{h}$ and $\phi_{h'}$ be partial automorphisms. We call  the following set-up an \textit{$(h,h',\phi_h,\phi_{h'},g)$ bad situation}: there exist vertices $x,x',y \in V$ so that
		
		\begin{enumerate}[label=(B\arabic*)]
			
			\item \label{prt:badend} $x \in N$ or $x=b(x,h \circ g)$,\\ $x' \in N$ or $x'=b(x',h' \circ g)$,\\ $y \in N$ or $y=e(y,h \circ g)=e(y,h' \circ g)$,
			\item \label{prt:badimage} $h^{-1}(x)=h'^{-1}(x')$,
			\item \label{prt:baddist} 
			\begin{enumerate}[ref=(B3.\alph*)]
				\item \label{prt:baddista} $x,y \in \dom(\phi_h)$, $x',y \in \dom(\phi_{h'})$
				\item \label{prt:baddistb} it is not true that  \[\phi_{h}(x) R f(\phi_{h}(y)) \iff \phi_{h'}(x') R f(\phi_{h'}(y)) .\]
			\end{enumerate}

		\end{enumerate}

		In case we would like to specify the roles of vertices, we will also call such a set-up an \textit{$(h,h',\phi_h,\phi_{h'},g,x,x',y)$ bad situation}, or when clear from the context, an \textit{$(h,h',x,x',y)$ bad situation.}
		
	\end{definition}

	\begin{definition}
		\label{d:uglysitu}
		Let $h,h' \in \mathcal{K}$ and $g$, $\phi_{h}$ and $\phi_{h'}$ be partial automorphisms. We call  the following set-up an \textit{$(h,h',\phi_h,\phi_{h'},g)$ ugly situation}: there exist vertices $x,y \in V$ so that $(h,h',\phi_h,\phi_{h'},x,y,y)$ has Properties \ref{prt:badend}, \ref{prt:badimage} of bad situations,
		\begin{enumerate}[label=(U\arabic*)]
			\item \label{prt:uglynew} $y \not \in \dom(\phi_{h'}),$
			\item \label{prt:uglydist} $x,y \in \dom(\phi_h)$ and $\phi_{h}(x) R f(\phi_{h}(y))$.
		\end{enumerate}

		We will use the conventions used at bad situations in the naming of ugly situations as well.
	\end{definition}
	
	Now we are ready to formulate our inductive assumptions. We will use the notations fixed in \ref{n:impo}.
	
	\begin{definition}
		\label{d:goodtriple}
		We say that the triple $(g,(\phi_h)_{h \in \mc{K}},M)$ is \textit{good} if the following conditions hold for every $h,h' \in \mc{K}$:
		\begin{enumerate}[label=(\roman*)]
			\item \label{prt:auto} $M$ is a finite set of vertices, $g$ and $\phi_h$ are partial automorphisms, 
			\item \label{prt:dom} $\dom(\phi_{h}) \supset \rd(h \circ g)$ and  $N \cup \rd(g) \cup \dom(\phi_{h})\subset M$, 
			\item \label{prt:nnoying} $N \subset \rd(g)$, $N \subset \dom(\phi_h)$, $g|_N=\phi_h|_N=id|_N$, 
			\item \label{prt:conj} $\phi_{h} \circ h \circ g= f \circ \phi_{h}$, i. e., whenever both of the sides of the equation are defined then they are equal,
			\item \label{prt:distorb} for vertices $w, w' \in \dom(\phi_h) \setminus N$ we have that $\mc{O}^{h \circ g}(w) \not = \mc{O}^{h \circ g}(w')$ implies $\mc{O}^{f}(\phi_h(w)) \not = \mc{O}^{f}(\phi_h(w'))$,
			\item \label{prt:forth} if $h|_{M^*}=h'|_{M^*}$ then $\phi_{h}=\phi_{h'}$,
			
			\item \label{prt:nocirc} if $w \in V \setminus N$ then for every $i \in \mathbb{Z} \setminus \{0\}$ we have $(h \circ g)^i(w) \not =w$, in particular, the functions $b(w,h \circ g), e(w,h \circ g)$ are defined for every $w \in V \setminus N$,
			\item \label{prt:onecase} for every $w \in \dom(\phi_h)$ if $f(\phi_h(w)) R \phi_h(w)$ and $h'^{-1}(w)=h^{-1}(w)$ hold then $w \in \dom(\phi_{h'})$,
			\item \label{prt:uglysitu}  there are no $(h,h',\phi_h,\phi_{h'},g)$ ugly situations,
			\item \label{prt:badsitu} there are no $(h,h',\phi_h,\phi_{h'},g)$ bad situations.

		\end{enumerate}
	\end{definition}
	We start the proof with a couple of trivial observations.
	
	\begin{remark}
		\label{rm:extm} It is easy to see that if $(g,(\phi_h)_{h \in \mc{K}},M)$ is a good triple and $\f{M} \supset M$ is finite then $(g,(\phi_h)_{h \in \mc{K}},\f{M})$ is also a good triple. 
	\end{remark}
	
	\begin{lemma}
		\label{l:ind0} $(id_N,(id_N)_{h \in \mathcal{K}},N)$ is a good triple.
	\end{lemma}
	\begin{proof}
		Properties \ref{prt:auto}-\ref{prt:onecase} are obvious. We check the remaining two properties:
		\begin{enumerate}[label=(\roman*)]
			\setcounter{enumi}{8}
			\item note that if $\dom(\phi_h)=\dom(\phi_{h'})$ then there are no $(h,h',\phi_h,\phi_{h'},g)$ ugly situations as the conjunction of \ref{prt:uglynew} and \ref{prt:uglydist} cannot be true,
			\item if we had an $(h,h',id_N,id_N,id_N,x,x',y)$ bad situation, then by property \ref{prt:baddista} we would have $x,x' \in N$ so by \ref{prt:badimage},  $h|_N=h'|_N=f|_N$ and the fact that $N$ is the union of orbits of $f$ clearly $x=x'$, but then \ref{prt:baddistb} could not be true.
		\end{enumerate}
		
	\end{proof}
	\begin{lemma}
		\label{l:firsttriv}
		Let $(g,(\phi_h)_{h \in \mc{K}},M)$ be a good triple and $h,h' \in \mc{K}$. Suppose that $g \subset \f{g}$, $\phi_h \subset \f{\phi}_h$ and $\phi_{h'} \subset \f{\phi}_{h'}$ are partial automorphisms. Suppose moreover that if $\phi_h \subsetneqq \f{\phi}_h$ then we have $\{v\}=\dom(\f{\phi}_h) \setminus \dom({\phi}_h)$ and similarly if $\phi_{h'} \subsetneqq \f{\phi}_{h'}$ then we have $\{v'\}=\dom(\f{\phi}_{h'}) \setminus \dom(\phi_{h'})$ so that 
		
		\begin{enumerate}[label=(\alph*)]
			\item \label{prt:weaksplit} $h^{-1}(v) \not =h'^{-1}(v)$,
			\item \label{prt:distance} $d_{\mc{K}}(v,M)>2$ and $d_{\mc{K}}(v',M)>2$
		\end{enumerate}
		then 
		\begin{enumerate}
			\item \label{lp:ugly} there are no $(h,h',\f{\phi}_h,\f{\phi}_{h'},\f{g})$ ugly situations.
			\item \label{lp:bad} if for some $x,x',y$ there exists an $(h,h',\f{\phi}_h,\f{\phi}_{h'},\f{g},x,x',y)$ bad situation then either
			\begin{itemize}
				\item $y \in M$, $x=v$ and $x'=v'$ OR
				\item $y=v=v'$ and $x,x' \in M$,
			\end{itemize}
			(in particular, $v$ and $v'$ must exist).
		\end{enumerate}

	\end{lemma}
	\begin{proof}
		
		First notice that in the definition of both ugly and bad situations the automorphism $g$ is only used in property \ref{prt:badend}, and this property does not use $\phi_h$ or $\phi_{h'}$. Moreover, by the definition of functions $b$ and $e$ clearly if $\f{g} \supset g$ and $x=b(x,h \circ \f{g})$ then  $x=b(x,h \circ g)$ and similarly for $e$. Therefore, if \ref{prt:badend} holds for $(h,h',\f{\phi}_h,\f{\phi}_{h'},\f{g},x, x',y)$ then it also holds for $(h,h',\phi_h,\phi_{h'},g,x, x',y)$. 
		
		Notice that using this observation about \ref{prt:badend} we can conclude that if $x,y \in \dom(\phi_h)$ and $x',y \in \dom(\phi_{h'})$ and there is an $(h,h',\f{\phi}_h,\f{\phi}_{h'},\f{g},x, x',y)$ then it is also an $(h,h',\phi_h,\phi_{h'},g,x, x',y)$ bad situation (and similarly with $x,y \in \dom(\phi_h)$ for ugly situations). So in order to prove the impossibility of an $(h,h',\f{\phi}_h,\f{\phi}_{h'},\f{g},x, x',y)$ bad situation, since $(g,(\phi_h)_{h \in \mc{K}},M)$ is a good triple it is enough to show that $x,y \in \dom(\phi_h)$ and $x',y \in \dom(\phi_{h'})$ (and analogously for ugly situations).
		
		Now we prove the statements of the lemma.

		(\ref{lp:ugly}) If there exists a  $(h,h',\f{\phi}_h,\f{\phi}_{h'},\f{g},x,y)$ ugly situation then by the above argument and the fact that we have (possibly) extended $\phi_h$ only to $v$ and $\phi_{h'}$ only to $v'$, we get $\{x,y\} \cap \{v,v'\} \not = \emptyset$. Moreover, from the definition of an ugly situation \ref{prt:badimage} holds for $x$ and $x'=y$ so clearly $d_{\mc{K}}(x,y) \leq 2$. This implies by assumption \ref{prt:distance} that $x, y \not \in M$. Using \ref{prt:uglydist} we obtain  $\{x, y\} \subset \dom(\f{\phi}_h) \setminus M$ and by Property \ref{prt:dom} of good triples $\dom(\phi_h) \subset M$, so $\{x,y\} \subset \dom(\f{\phi}_h) \setminus \dom(\phi_h)=\{v\}$. But \ref{prt:badimage} gives that $h^{-1}(x)=h'^{-1}(y)$, so $h^{-1}(v)=h'^{-1}(v)$ contradicting the assumption \ref{prt:weaksplit} of the lemma.
		
		Now we prove (\ref{lp:bad}). 
		
		Suppose $y \not \in M$. Then by \ref{prt:baddista} we have $y \in \dom(\f{\phi}_h) \cap \dom(\f{\phi}_{h'}) \setminus M$, which is only possible using Property \ref{prt:dom} of good triples if $y=v=v'$. Since $x \in \dom(\f{\phi}_{h})$ clearly, $x \not \in M$ can happen only if $x=v=y$. Then, by \ref{prt:badimage} we have $d_{\mc{K}}(x,x') \leq 2$, so $x' \in  \dom(\f{\phi}_h) \setminus M$, therefore $x'=v'=y$. But then, using again  \ref{prt:badimage} we get $h^{-1}(v)=h^{-1}(x)=h'^{-1}(x')=h'^{-1}(v)$, contradicting \ref{prt:weaksplit}. So $x \in M$ and a similar argument shows $x' \in M$. 
		
		So assume $y \in M$, in particular by \ref{prt:baddista} and the assumptions of the lemma $y \in \dom(\phi_{h}) \cap \dom(\phi_{h'})$. Suppose now that $x \not = v$ (with the possibility that $v$ does not exists). Since by \ref{prt:baddista} we have $x \in \dom(\f{\phi}_h)$ and $\dom(\f{\phi}_h) \subset \{v\} \cup M$ clearly $x \in M$. Using property \ref{prt:badimage} and \ref{prt:baddista} we get $d_\mc{K}(x,x') \leq 2$ and $x' \in \dom(\f{\phi}_{h'})$ but by assumption \ref{prt:distance} this can happen only if $x' \in M$, so $x' \in \dom(\phi_{h'})$. Therefore $x,y \in \dom(\phi_{h})$ and $x', y \in \dom(\phi_{h'})$ which is impossible. Thus, $x=v$ and similarly $x'=v'$.
	\end{proof}
	Now we prove a lemma which ensures that a good triple can be extended. 
	\begin{lemma}
		\label{l:forward}
		Suppose that $(g,(\phi_h)_{h \in \mc{K}},M)$ is a good triple and $v \in \bigcap_{h \in \mathcal{K}}\dom(\phi_{h})$. Then there exist extensions $\f{g} \supset g$, $\f{\phi}_{h} \supset \phi_{h}$ and $\overline{M} \supset M$ so that $(\f{g}, (\f{\phi}_h)_{h \in \mc{K}},\f{M})$ is a good triple and $v \in \dom(\f{g})$. 
	\end{lemma}
	
	\begin{proof}
		We will find a suitable vertex $\f{v}$ and let $\f{g}=g \cup \langle v,\f{v} \rangle$.
		
		Define a map $\tau_g:\ran(g) \to 2$ as follows:
		\begin{equation}
		\label{e:randdef1}
		\tau_g(w)=1 \iff g^{-1}(w) R v,
		\end{equation}
		and maps $\tau_h:h^{-1}(\dom(\phi_h)) \to 2$ for each $h \in \mathcal{K}$ as
		\begin{equation}
		\label{e:randdef2}
		\tau_h(w)=1 \iff \phi_h(h(w)) R f(\phi_h(v)).
		\end{equation}
		\begin{claim} The maps $\tau_g, (\tau_h)_{h \in \mathcal{K}}$ are compatible, i. e., $\tau=\tau_g \cup \bigcup_{h \in \mathcal{K}} \tau_h$ is a function. 
		\end{claim}
		\begin{proof}[Proof of the Claim.]
			
			\textit{$\tau_g$ and $\tau_h$ are compatible.} Let $w \in \ran(g)=\dom(\tau_g)$ and let $h \in \mathcal{K}$ be arbitrary. Clearly, $g^{-1}(w) \in \dom(h \circ g) \subset \dom(\phi_h)$  and $(h \circ g)(g^{-1}(w)) \in \ran(h \circ g) \subset \dom(\phi_h)$ by Property \ref{prt:dom} of good triples. Therefore, we can use Property \ref{prt:conj} for $g^{-1}(w)$ (that is, in the following equation both of the sides are defined): \[(\phi_{h} \circ h \circ g)(g^{-1}(w))= (f \circ \phi_{h})(g^{-1}(w))\]
			so we get
			\begin{equation}
			\label{e:rand2}
			f^{-1}(\phi_{h} (h(w)))=   \phi_{h}(g^{-1}(w)).
			\end{equation}
			
			As $f$ is an automorphism
			\begin{equation}
			\label{e:rand1}
			\tau_h(w)=1 \iff \phi_h(h(w)) R f(\phi_h(v)) \iff f^{-1}(\phi_h(h(w))) R \phi_h(v).
			\end{equation}
			
			So by \eqref{e:rand2}, \eqref{e:rand1} and the fact that $\phi_h$ is a partial automorphism we have
			\[\tau_h(w)=1 \iff \phi_{h}(g^{-1}(w)) R \phi_h(v) \iff g^{-1}(w) R v.\]
			Comparing this equation to the definition of $\tau_g$ we obtain that $\tau_g$ and $\tau_h$ are indeed compatible.
			
			\textit{$\tau_h$ and $\tau_{h'}$ are compatible.} Now, using the first case it is enough to check compatibility for $w \not \in \ran(g)$. We will use Property \ref{prt:badsitu}, that there are no bad situations. Let us consider the sequence $(h,h',\phi_h,\phi_{h'},g,h(w),h'(w),v)$. Clearly, since $w \not \in ran(g)$, we have $h(w) \not \in \ran(h \circ g)$ thus $b(h(w),h \circ g)=h(w)$ and similarly $b(h'(w),h' \circ g)=h'(w)$. Moreover, as $v \not \in \dom(g)$ we have $e(v,h \circ g)=e(v,h' \circ g)=v$, so Property \ref{prt:badend} of bad situations hold. Moreover,  $h^{-1}(h(w))=w=h'^{-1}(h'(w))$, therefore Property \ref{prt:badimage} is also true. Clearly, by the assumptions of Lemma \ref{l:forward} we have $v,h(w) \in \dom(\phi_h)$ and $v,h'(w) \in \dom(\phi_{h'})$. Hence, as there are no bad situations Property \ref{prt:baddistb} must fail, consequently
			\[\phi_{h}(h(w)) R f(\phi_{h}(v)) \iff  \phi_{h'}(h'(w)) R f(\phi_{h'}(v)),  \]
			so, using this and the definition of $\tau_h$ and $\tau_{h'}$ we get
			\[\tau_h(w)=1 \iff \phi_h(h(w)) R f(\phi_h(v)) \iff\]\[  \phi_{h'}(h'(w)) R f(\phi_{h'}(v)) \iff \tau_{h'}(w)=1.\]
			
			This finishes the proof of the claim.
		\end{proof}
		
		Now we return to the proof of Lemma \ref{l:forward}. By Corollary \ref{c:splitting} there exists a splitting point $\f{v}$ for $M^*$ and $\mc{K}$ that realizes $\tau$ and $d_{\mc{K}}(\f{v},M^*)>3$ (in particular, by $M \subset M^*$ we have $d_{\mc{K}}(\f{v},M)>3$) extending $\tau$ to the whole $M^*$ arbitrarily if necessary. Let $\f{g}=g \cup \langle v,\f{v}\rangle$, $\f{M}=M \cup \{\f{v},h(\f{v}): h \in \mc{K}\}$ and for every $h \in \mc{K}$ let $\f{\phi}_h=\phi_h \cup \langle h(\f{v}), f(\phi_h(v)) \rangle$.
		
		We claim that $(\f{g}, (\f{\phi}_h)_{h \in \mc{K}},\f{M})$ is a good triple.
		
		\begin{enumerate}[label=(\roman*)]
			\item  By compactness $\f{M}$ is finite. We check that $\f{g}$ and $\f{\phi}_h$ are partial automorphisms. Since $d_{\mc{K}}(\f{v},M)>3$ and Property \ref{prt:dom} of good triples $\ran(g) \subset M$ so the function $\f{g}$ is injective. 
			
			We check the injectivity of the functions $\f{\phi}_h$. If for some $w$ we have 
			\begin{equation}
			\label{e:trivforw}
			\f{\phi}_h(h(\f{v}))=f(\phi_h(v))=\phi_{h}(w)
			\end{equation} then using the facts that $\phi_h|_N=id|_N$ and that $N$ is the union of the finite orbits of $f$ we can conclude that $w \in N$ would imply $\phi_h(v) \in N$, so $v \in N \subset \dom(g)$ which is impossible. So $w \not \in N$ and also $\phi_h(w) \not \in N$. 
			By \eqref{e:trivforw} we have that $\mc{O}^{f}(\phi_h(v))=\mc{O}^{f}(\phi_h(w))$ and clearly, $v \not \in N$, so using Property \ref{prt:distorb} of good triples we obtain $\mc{O}^{h \circ g}(v)=\mc{O}^{h \circ g}(w)$. Then as $v \not \in \dom(g)$ clearly $v=(h \circ g)^{k}(w)$ for some $k \geq 0$. Suppose $k>0$. Applying $\phi_h$ to both sides and using \ref{prt:conj} of good triples we get
			\[\phi_h(v)=\phi_h((h \circ g)^{k}(w))=f(\phi_h(h \circ g)^{k-1}(w))=\dots=f^k(\phi_h(w)),\]
			but then $f(\phi_h(v))=f^{k+1}(\phi_h(w))$ therefore $f^{k+1}(\phi_h(w))=\phi_h(w)$, contradicting the fact that $f$ has only infinite orbits outside of $N$. Thus $k=0$ and $v=w$, so $\f{\phi}_h$ is indeed injective.

			So we only have to check $\f{g}$ and $\f{\phi}_h$ preserve the relation, that is, for every $w,w' \in \dom(\f{g})$ distinct we have 
			\[wRw' \iff  \f{g}(w)R\f{g}(w'),\]
			and it is enough to check this condition if $\{w,w'\}    \not \subset \dom(g)$ (and similarly for $\phi_h$). 
			So suppose that $w\in \dom(g)$ and $w' \in \dom(\f{g}) \setminus \dom(g)$, that is, $w'=v$. Then by the fact that $g(w) \in \ran(g) =\dom(\tau_g)$, \eqref{e:randdef1} and the definition of $\tau$ we have
			\[\f{g}(w')R\f{g}(w) \iff \f{g}(v)R\f{g}(w) \iff \f{v}Rg(w) \iff \tau(g(w))=1 \iff\]\[  \tau_g(g(w))=1\iff g^{-1}(g(w))Rv \iff wRv \iff wRw',\]
			so indeed, $\f{g}$ preserves the relation.
			
			Now if $w \in \dom(\phi_h)$ and $w' \in \dom(\f{\phi}_h) \setminus \dom(\phi_h)$, that is, $w'=h(\f{v})$ then $h^{-1}(w) \in h^{-1}(\dom(\phi_h)) =\dom(\tau_h)$. Then we have
			\[\f{\phi}_h(w)R\f{\phi}_h(w') \iff \phi_h(w)R\f{\phi}_h(h(\f{v})) \]
			which is by the definition of $\f{\phi}_h$
			\[ \iff \phi_h(w)Rf(\phi_h(v))  \iff \phi_h(h(h^{-1}(w)))Rf(\phi_h(v))\]
			using the definition of $\tau$ and \eqref{e:randdef2} we get
			\[  \iff \tau_h(h^{-1}(w))=1\iff h^{-1}(w)R\f{v} \iff wRh(\f{v}) \iff wRw',\]
			so we are done.
			
			\item By the definition of $\f{\phi}_h$ we have $\dom(\f{\phi}_h) = \dom(\phi_h) \cup \{h(\f{v})\} \supset \rd(h \circ g) \cup \{v,h(\f{v})\}$ and using the fact that $d_{   \mc{K}}(\f{v},\dom(\f{g}))>3$ we obtain that $h(\f{v}) \not \in \dom(g)$, thus $ \rd(h \circ g) \cup \{v,h(\f{v})\}= \rd(h \circ \f{g})$. Moreover, $\rd(\f{g}) \cup  \dom(\f{\phi}_h) \subset \f{M}$. 
			\item Obvious.
			\item It is enough to check equality $(\f{\phi}_{h} \circ h \circ \f{g})(v_0)= (f \circ \f{\phi}_{h})(v_0)$ for $v_0=v$, as for $v_0 \in \dom(g)$ we have $v_0 \in \dom(g) \subset \dom(\phi_{h})$ and $h(g(v_0)) \in \rd(h \circ g) \subset \dom(\phi_{h})$ so the equality holds because we started with a good triple.  But using the definition of $\f{\phi}_h$ we have $\f{\phi}_h(h(\f{g}(v)))=\f{\phi}_h(h(\f{v}))=f(\f{\phi}_{h}(v))$.
			\item  Suppose that for vertices $w, w' \in \dom(\f{\phi}_h) \setminus N$ we have that $\mc{O}^{h \circ \f{g}}(w) \not = \mc{O}^{h \circ \f{g}}(w')$. Then of course $w \not = w'$. We have extended $\phi_h$ only to $h(\f{v})$ so it is enough to check the property with $w=h(\f{v})$ and $w' \in \dom(\phi_h)$. But $\mc{O}^{h \circ \f{g}}(h(\f{v}))=\mc{O}^{h \circ \f{g}}(h(\f{g}(v)))=\mc{O}^{h \circ \f{g}}(v)$, thus, $\mc{O}^{h \circ g}(v) \cap \mc{O}^{h \circ g}(w')\subset \mc{O}^{h \circ \f{g}}(h(\f{v})) \cap \mc{O}^{h \circ \f{g}}(w')= \emptyset$. Therefore, by the fact that $h(\f{v}) \not \in N$ implies $v \not \in N$ and Property \ref{prt:distorb} of good triples we get $\mc{O}^{f}(\phi_h(v)) \cap \mc{O}^{f}(\phi_h(w'))=  \emptyset$. Using this and the definition of $\f{\phi}_h$ we have $\mc{O}^{f}(\f{\phi}_h(h(\f{v})))=\mc{O}^{f}(\phi_h(v))$ and $\mc{O}^{f}(\f{\phi}_h(v)) \cap \mc{O}^{f}(\f{\phi}_h(w'))=\mc{O}^{f}(\phi_h(v)) \cap \mc{O}^{f}(\phi_h(w'))=\emptyset$ so we are done.
			\item If $h,h' \in \mathcal{K}$ and $h|_{\f{M}^*}=h'|_{\f{M}^*}$ then in particular $h(\f{v})=h'(\f{v})$ and $h|_{M^*}=h'|_{M^*}$ so $\phi_{h}=\phi_{h'}$. Then by definition $\f{\phi}_{h}=\f{\phi}_{h'}$.
			
			\item If $h \in \mathcal{K}$ and $w \not \in N$ is a vertex and for some $i \in \omega \setminus \{0\}$ we have $(h \circ \f{g})^i(w)=w$ then at least one of the points $\{w,\dots, (h \circ \f{g})^{i-1}(w)\}$ is not in the domain of $g$, otherwise the triple $(g,(\phi_h)_{h \in \mc{K}},M)$ would already violate this property of good conditions. In other words, $v \in \{w,\dots, (h \circ \f{g})^{i-1}(w)\}$. 
			Moreover, $\{w,\dots, (h \circ \f{g})^{i-1}(w)\} \subset \dom(\f{g})$ and clearly $(h \circ \f{g})(\{w,\dots, (h \circ \f{g})^{i-1}(w)\})=\{w,\dots, (h \circ \f{g})^{i-1}(w)\}$, so $(h\circ \f{g})(v) \in \dom(\f{g})$, that is, $h(\f{v}) \in \dom(\f{g})$. But we know that $\dom(\f{g})=\{v\} \cup \dom(g) \subset \cap_h \dom(\phi_h) \cup \dom(g) \subset M$. Therefore $h(\f{v}) \in M$, contradicting the assumption that $d_{\mc{K}}(\f{v},M)>3$.
			\item Since $(g, (\phi_h)_{h \in \mc{K}},M)$ is a good triple, this condition can fail for $(\f{g}, (\f{\phi}_h)_{h \in \mc{K}},\f{M})$, an $h$ and $w$ only if $w \in \dom(\f{\phi}_{h}) \setminus \dom(\phi_{h})$, in other words $w=h(\f{v})$. 
			So suppose $h^{-1}(w)=h'^{-1}(w)$, that is, $\f{v}=h^{-1}(h(\f{v}))=h'^{-1}(h(\f{v}))$. This means that $h'(\f{v})=h(\f{v})$, but then by \ref{prt:dom} we have $w=h'(\f{v}) \in \dom(\f{\phi}_{h'})$ as well.
			\item Suppose that there exists an $(h,h',\f{\phi}_h,\f{\phi}_{h'},x,y)$ ugly situation. If $h|_{M^*}=h'|_{M^*}$ and $h(\f{v})=h'(\f{v})$ then $\f{\phi}_h=\f{\phi}_{h'}$ which contradicts properties \ref{prt:uglynew} and \ref{prt:uglydist}. 
			
			Now if $h|_{M^*} \not =h'|_{M^*}$ or $h(\f{v}) \not =h'(\f{v})$, then since $\f{v}$ is a splitting point we have $h(\f{v}) \not =h'(\f{v})$. Consequently, $h^{-1}(h(\f{v})) \not =h'^{-1}(h(\f{v}))$ and also by $d_{\mc{K}}(\f{v},M)>3$ clearly $d_{\mc{K}}(h(\f{v}),M) >2$ and $d_{\mc{K}}(h'(\f{v}),M) >2$. Then we can apply Lemma \ref{l:firsttriv} for $(h,h',\f{\phi}_h,\f{\phi}_{h'},x,y)$ and $\{h(\f{v}) \}\supset \dom(\f{\phi}_h) \setminus \dom(\phi_h)$, $\{h'(\f{v})\} \supset \dom(\f{\phi}_{h'}) \setminus \dom(\phi_{h'})$ by which there are no $(h,h',\f{\phi}_h,\f{\phi}_{h'},x,y)$ ugly situations.
			
			\item  Here the argument is similar. Suppose that there exists an $(h,h',\f{\phi}_h,\f{\phi}_{h'},x,x',y)$ bad situation. If $h|_{M^*}=h'|_{M^*}$ and $h(\f{v})=h'(\f{v})$ then $\f{\phi}_h=\f{\phi}_{h'}$ but then \ref{prt:baddist} cannot be true. Now if $h|_{M^*} \not =h'|_{M^*}$ then as in the previous point we can use Lemma \ref{l:firsttriv} for $(h,h',\f{\phi}_h,\f{\phi}_{h'},x,x',y)$. Therefore, either $x=h(\f{v})$, $x'=h'(\f{v})$ or $y=h(\f{v})=h'(\f{v})$. But the second option is impossible since $y=h(\f{v})=h'(\f{v})$ contradicts that $\f{v}$ was a splitting point. Now the first option is also impossible unless $x , x' \in N$: as $\f{v}=\f{g}(v)$, that is, $x=h(\f{g}(v))$, contradicting \ref{prt:badend}. But if $x,x' \in N \subset M$ then $d_{\mc{K}}(\f{v},M) \leq 1$, a contradiction again.

		\end{enumerate}

	\end{proof}
	
	Now we prove a lemma which allows us to extend $g$ backwards. The proof is very similar to the proof of the forward extension, although to treat both cases in the same framework would have a great technical cost. For the sake of completeness we write down the proofs in detail.
	
	\begin{lemma}
		Suppose that $(g, (\phi_h)_{h \in \mc{K}},M)$ is a good triple and $v \in M$ is a vertex so that for every $h \in \mathcal{K}$ we have $h(v) \in \dom(\phi_h)$. Then there exist extensions $\f{g} \supset g$, $\f{\phi}_{h} \supset \phi_{h}$ and $\overline{M} \supset M$ so that $(\f{g}, (\f{\phi}_h)_{h \in \mc{K}},\f{M})$ is a good triple and $v \in \ran(\f{g})$.
		\label{l:backward}
	\end{lemma}
	
	\begin{proof}
		We will find a suitable vertex $\f{v}$ and let $\f{g} =g \cup \langle \f{v},v\rangle$.
		
		Define a map $\tau_g:\dom(g) \to 2$ as follows:
		\begin{equation}
		\label{e:randdefback1}
		\tau_g(w)=1 \iff g(w) R v,
		\end{equation}
		and maps $\tau_h:\dom(\phi_h) \to 2$ for each $h \in \mathcal{K}$
		\begin{equation}
		\label{e:randdefback2}
		\tau_h(w)=1 \iff \phi_h(w) R f^{-1}(\phi_h(h(v))).
		\end{equation}
		\begin{claim} The maps $\tau_g, (\tau_h)_{h \in \mathcal{K}}$ are compatible, i. e., $\tau=\tau_g \cup \bigcup_{h \in \mathcal{K}} \tau_h$ is a function. 
		\end{claim}
		\begin{proof}[Proof of the Claim.]
			\textit{$\tau_g$ and $\tau_h$ are compatible.} Let $h \in \mc{K}$ be arbitrary and $w \in \dom(\tau_g) \cap \dom(\tau_h)=\dom(g) \cap \dom(\tau_h)$. Clearly, by Property \ref{prt:dom} of good triples we have $w,h(g(w)) \in \dom(\phi_h)$. So we can use Property \ref{prt:conj} for $w$ and we get 
			\[(\phi_{h} \circ h \circ g)(w)= (f \circ \phi_{h})(w).\]
			From the definition of $\tau_h$ we obtain
			\[
			\tau_h(w)=1 \iff \phi_h(w) R f^{-1}(\phi_h(h(v))) \iff f(\phi_h(w)) R \phi_h(h(v)).
			\]
			Putting together these equations and using that $\phi_h$ is an automorphism we obtain
			\[\tau_h(w)=1 \iff (\phi_{h} \circ h \circ g)(w)  R \phi_h(h(v)) \iff g(w) R v \iff \tau_g(w)=1.\]

			\textit{$\tau_h$ and $\tau_h'$ are compatible.} Let $h,h' \in \mathcal{K}$ be arbitrary and $w \in \dom(\tau_{h}) \cap \dom(\tau_{h'})$. By the fact that $\tau_h$ and $\tau_h'$ are compatible with $\tau_g$ we can assume $w \not \in \dom(\tau_g)=\dom(g).$
			
			We will use Property \ref{prt:badsitu}, that there are no bad situations. Let us consider the sequence $(h,h',\phi_h,\phi_{h'},g,h(v),h'(v),w)$. Clearly, by $v \not \in \ran(g)$ we have  $b(h(v),h \circ g)=h(v)$ and similarly $b(h'(v),h' \circ g)=h'(v)$. Moreover, as $w \not \in \dom(g)$, we have $e(w,h \circ g)=e(w,h' \circ g)=w$, so Property \ref{prt:badend} of Definition \ref{d:badsitu} holds. Obviously, $h^{-1}(h(v))=h'^{-1}(h'(v))$, therefore Property \ref{prt:badimage} is also true. By the assumptions of Lemma \ref{l:backward} clearly $h(v),w \in \dom(\phi_h)$ and $h'(v),w \in \dom(\phi_{h'})$, hence, as there are no bad situations Property \ref{prt:baddistb} must fail, consequently
			\[\phi_{h}(h(v)) R f(\phi_{h}(w)) \iff  \phi_{h'}(h'(v)) R f(\phi_{h'}(w)),  \]
			
			so, by definition of $\tau_h$ and $\tau_{h'}$ we get
			\[\tau_h(w)=1 \iff f^{-1}(\phi_h(h(v))) R \phi_h(w) \iff \phi_h(h(v)) R f(\phi_h(w)) \iff\]\[  \phi_{h'}(h'(v)) R f(\phi_{h'}(w)) \iff f^{-1}(\phi_{h'}(h'(v))) R \phi_h(w) \iff \tau_{h'}(w)=1.\]
			
			This finishes the proof of the claim. 
		\end{proof}

		Now we return to the proof of Lemma \ref{l:backward}. By Corollary \ref{c:splitting} there exists a splitting point $\f{v}$ for $M^*$ and $\mc{K}$ that realizes $\tau$ and $d_{\mc{K}}(\f{v},M^*)>3$. Let $\f{g}=g \cup \langle \f{v},v\rangle$, $\f{M}=M \cup \{\f{v}\}$ and for every $h \in \mc{K}$ let $\f{\phi}_h=\phi_h \cup \langle \f{v}, f^{-1}(\phi_h(h(v))) \rangle$.
		
		We claim that $(\f{g}, (\f{\phi}_h)_{h \in \mc{K}},\f{M})$ is a good triple.
		
		\begin{enumerate}[label=(\roman*)]
			\item  We check that $\f{g}$ and $\f{\phi}_h$ are partial automorphisms. Since $d_{\mc{K}}(\f{v},M)>2$, $\dom(g) \subset M$ and $d_{\mc{K}}(\f{v},M)>2$ the function $\f{g}$ is injective. 
			
			We check the injectivity of functions  $\f{\phi}_h$.  
			
			If for some $w$ we have \begin{equation} \label{e:trivback} \f{\phi}_h(\f{v})=f^{-1}(\phi_h(h(v)))=\phi_{h}(w) \end{equation} then using the facts that $\phi_h|_N=id|_N$ and that $N$ is the union of the finite orbits of $f$ we can conclude again that $w \in N$ would imply $\phi_h(h(v)) \in N$, and thus $v, h(v) \in N \subset \ran(g)$ which is impossible. So $w,h(v) \not \in N$. But by \eqref{e:trivback} we have $\mc{O}^f(\phi_h(h(v)))=\mc{O}^f(\phi_{h}(w))$ so using Property \ref{prt:distorb} of good triples we obtain $\mc{O}^{h \circ g}(h(v))=\mc{O}^{h \circ g}(w)$. Then as $v \not \in \ran(g)$ clearly $w=(h \circ g)^{k}(h(v))$ for some $k \geq 0$. Suppose $k>0$. Applying $\phi_h$ to both sides and using Property \ref{prt:conj} of good triples we get
			\[\phi_h(w)=\phi_h((h \circ g)^{k}(h(v)))=\]
			\[=f(\phi_h((h \circ g)^{k-1}(h(v))))=\dots=f^k(\phi_h(h(v))),\]
			but then $f(\phi_h(w))=f^{k+1}(\phi_h(h(v)))$, therefore by \eqref{e:trivback} we get $f^{k+1}(\phi_h(h(v)))=\phi_h(h(v))$, contradicting the fact that $f$ has only infinite orbits outside of $N$. Thus $k=0$ and $\f{v}=w$, so $\f{\phi}_h$ is indeed injective.

			We have to check $\f{g}$ preserves the relation, and again it is enough to check for $w\in \dom(g)$ and $w' \in \dom(\f{g}) \setminus \dom(g)$, that is, $w'=\f{v}$. Then by the fact that $w \in \dom(g) =\dom(\tau_g)$, \eqref{e:randdef1} and the definition of $\tau$ we have
			\[\f{g}(w')R\f{g}(w) \iff \f{g}(\f{v})R\f{g}(w) \iff vRg(w) \iff \tau_g(w)=1 \iff\]\[  wR\f{v} \iff wRw',\]
			so indeed, $\f{g}$ preserves the relation.
			
			Now if $w \in \dom(\phi_h)$ and $w' \in \dom(\f{\phi}_h) \setminus \dom(\phi_h)$, that is, $w'=\f{v}$ then we have
			\[\f{\phi}_h(w)R\f{\phi}_h(w') \iff \phi_h(w)R\f{\phi}_h(\f{v}) \]
			which is by the definition of $\f{\phi}_h$, $\tau$ and \eqref{e:randdefback2}
			\[ \iff \phi_h(w)Rf^{-1}(\phi_h(h(v)))  \iff \tau_h(w)=1\iff wR\f{v},\]
			so we are done.
			
			\item By the definition of $\f{\phi}_h$ we have $\dom(\f{\phi}_h) = \dom(\phi_h) \cup \{\f{v}\}\supset \rd(h \circ g) \cup \{\f{v}\}$ and using the fact that $d_{\mc{K}}(\f{v},\dom(\f{g}))>3$ we obtain that $h^{-1}(\f{v}) \not \in \dom(g)$, thus $ \rd(h \circ g) \cup \{\f{v}\}= \rd(h \circ \f{g})$. Clearly, $\rd(\f{g}) \cup  \dom(\f{\phi}_h) \subset \f{M}$. 
			\item Obvious.
			\item It is enough to check equality $\f{\phi}_{h} \circ h \circ \f{g}(v_0)= f \circ \f{\phi}_{h}(v_0)$ for $v_0=\f{v}$, as for $v_0 \in \dom(g)$ we have $v_0 \in \dom(g) \subset \dom(\phi_{h})$ so the equality holds because we started with a good triple.  But using the definition of $\f{\phi}_h$ and the fact that $h(v) \in \dom(\phi_h)$ we have $\f{\phi}_h(h(\f{g}(\f{v})))=\f{\phi}_h(h(v))=\phi_h(h(v))=f(\f{\phi}_{h}(v))$.
			\item  Suppose that for vertices $w, w' \in \dom(\f{\phi}_h) \setminus N$ we have that $\mc{O}^{h \circ \f{g}}(w) \not = \mc{O}^{h \circ \f{g}}(w')$. Then of course $w \not = w'$. We have extended $\phi_h$ only to $\f{v}$ so it is enough to check the property with $w=\f{v}$ and $w' \in \dom(\phi_h)$. But $\mc{O}^{h \circ \f{g}}(\f{v})=\mc{O}^{h \circ \f{g}}(h(\f{g}(\f{v})))=\mc{O}^{h \circ \f{g}}(h(v))$ thus $\mc{O}^{h \circ g}(h(v)) \cap \mc{O}^{h \circ g}(w')\subset \mc{O}^{h \circ \f{g}}(\f{v}) \cap \mc{O}^{h \circ \f{g}}(w')= \emptyset$. Therefore, by the facts that $\f{v} \not \in N$ implies $h(v) \not \in N$ and we started with a good triple, by Property \ref{prt:distorb} we obtain $\mc{O}^{f}(\phi_h(h(v))) \cap \mc{O}^{f}(\phi_h(w'))=  \emptyset$. Using this and the definition of $\f{\phi}_h$ we have $\mc{O}^{f}(f^{-1}(\f{\phi}_h(h(v))))=\mc{O}^{f}(\f{\phi}_h(\f{v}))$ thus  $\mc{O}^{f}(\f{\phi}_h(\f{v})) \cap \mc{O}^{f}(\f{\phi}_h(w'))=\mc{O}^{f}(\phi_h(h(v))) \cap \mc{O}^{f}(\phi_h(w'))=\emptyset$ so we are done.
			\item If $h,h' \in \mathcal{K}$ and $h|_{\f{M}^*}=h'|_{\f{M}^*}$ then $h|_{M^*}=h'|_{M^*}$ thus $\phi_{h}=\phi_{h'}$, so by definition $\f{\phi}_{h}=\f{\phi}_{h'}$.
			
			\item If $h \in \mathcal{K}$, and for some $i \in \omega \setminus \{0\}$ we have $(h \circ \f{g})^i(w)=w$ then at least one of the points $\{w,\dots, (h \circ \f{g})^{i-1}(w)\}$ is not in the domain of $g$, otherwise the triple $(g,(\phi_h)_{h \in \mc{K}},M)$ would violate this property of good conditions. In other words, $\f{v} \in \{w,\dots, (h \circ \f{g})^{i-1}(w)\}$. 
			Moreover, $\{w,\dots, (h \circ \f{g})^{i-1}(w)\} \subset \dom(\f{g})$ and clearly $(h \circ \f{g})^{-1}(\{w,\dots, (h \circ \f{g})^{i-1}(w)\})=\{w,\dots, (h \circ \f{g})^{i-1}(w)\}$, so $\f{v} \in \ran(h\circ \f{g})$, that is, $h^{-1}(\f{v}) \in \ran(\f{g})$. But we know that $\ran(\f{g})=\{v\} \cup \ran(g) \subset M$. Therefore $h(\f{v}) \in M$, contradicting the assumption that $d_{\mc{K}}(\f{v},M)>3$.
			\item Since $(g, (\phi_h)_{h \in \mc{K}},M)$ is a good triple, this condition can fail for $(\f{g}, (\f{\phi}_h)_{h \in \mc{K}},\f{M})$, an $h$ and $w$ only if $w \in \dom(\f{\phi}_{h}) \setminus \dom(\phi_{h})$, in other words $w=\f{v}$. But $\f{v} \in \dom(\f{\phi}_{h'})$ for every $h' \in \mc{K}$ as well.
			
			\item Suppose that there exists an $(h,h',\f{\phi}_h,\f{\phi}_{h'},x,y)$ ugly situation. If $h|_{M^*}=h'|_{M^*}$ then $\f{\phi}_h=\f{\phi}_{h'}$ which contradicts properties \ref{prt:uglynew} and \ref{prt:uglydist}. 
			
			Now if $h|_{M^*} \not =h'|_{M^*}$ then since $\f{v}$ is a splitting point we have $h^{-1}(\f{v}) \not =h'^{-1}(\f{v})$ and also $d_{\mc{K}}(\f{v},M)>2.$ Then we can apply Lemma \ref{l:firsttriv} for $(h,h',\f{\phi}_h,\f{\phi}_{h'},x,y)$ and $\{\f{v}\} \supset \dom(\f{\phi}_h) \setminus \dom(\phi_h)$, $\{\f{v}\} \supset \dom(\f{\phi}_{h'}) \setminus \dom(\phi_{h'})$ so there are no $(h,h',\f{\phi}_h,\f{\phi}_{h'},x,y)$ ugly situations.
			
			\item  Suppose that there exists an $(h,h',\f{\phi}_h,\f{\phi}_{h'},x,x',y)$ bad situation. If $h|_{M^*}=h'|_{M^*}$ then $\f{\phi}_h=\f{\phi}_{h'}$ but then \ref{prt:baddist} cannot be true. 
			
			Now if $h|_{M^*} \not =h'|_{M^*}$, then as above we can use Lemma \ref{l:firsttriv} for $(h,h',\f{\phi}_h,\f{\phi}_{h'},x,x',y)$. Therefore, either $x=x'=\f{v}$ or $y=\f{v}$. The first option is impossible, as by \ref{prt:badimage} we would obtain $h^{-1}(\f{v})=h^{-1}(x)=h'^{-1}(x')=h'^{-1}(\f{v})$ contradicting the fact that $\f{v}$ was a splitting point. We can exclude the second option, as $\f{v} \in \dom(\f{g})$, so we have $e(y,h \circ g) \not = y$ thus using  \ref{prt:badend} we get $y \in N$, which is impossible again by $d_{\mc{K}}(\f{v},M)>3.$ 
		\end{enumerate}
	\end{proof}
	Now we prove a lemma which is the essence of the proof, namely that we can extend the maps $\phi_h$ forward as well.
	\begin{lemma}
		\label{l:extend}
		Suppose that $(g, (\phi_h)_{h \in \mc{K}},M)$ is a good triple, $h \in \mc{K}$ and $v \in M.$ Then there exists a vertex $z$ so that if for every $h' \in \mc{K}$ with $h'|_{M^*}=h|_{M^*}$ we extend $\phi_{h'}$ by letting $\f{\phi}_{h'}=\phi_{h'} \cup \langle v,z\rangle$ then $(g, (\phi_{h'})_{h' \in \mathcal{K}, h'|_{M^*} \not =h|_{M^*}} \cup (\f{\phi}_{h'})_{h \in \mathcal{K}, h'|_{M^*} =h|_{M^*}},M)$ is a good triple.
	\end{lemma}
	
	\begin{proof}
		
		First find a vertex $z$ satisfying the following requirements (note that the below requirements depend solely on $h|_{M^*}$, hence these will be exactly the same for every $h' \in \mc{K}$ so that $h'|_{M^*}=h|_{M^*}$):
		
		\begin{enumerate}[label=(\arabic*)]
			\item \label{prt:extfirst} $z \lnot R f(z)$ and $z \not \in \mc{O}^f(\ran(\phi_h))$,
			\item \label{prt:extauto} for every $w \in \dom(\phi_h)$ we have $z R \phi_h( w) \iff v R w$,
			\item \label{prt:extbadfirst} if for some $h' \in \mc{K}$ and $x,x' \in V$ the sequence $(h,h',x,x',v)$ has Properties \ref{prt:badend} and \ref{prt:badimage} of a bad situation then 
			\[\tag{z.3.B}
			\label{e:3zB} \text{ if } x \in \dom(\phi_h) \text{ and } x',v \in \dom(\phi_{h'})\text{  holds then }\]
			\[z R f^{-1}(\phi_h(x)) \iff   f(z) R \phi_h(x) \iff \phi_{h'}(x') R f(\phi_{h'}(v))\]
			i. e., \ref{prt:baddistb} is false with $z=\f{\phi}_h(v)$,
			\[ \text{ if } x \in \dom(\phi_h), v=x' \text{ and } v \not \in \dom(\phi_{h'}) \text{ holds then } z \lnot R f^{-1}(\phi_{h}(x)), 
			\tag{z.3.U}
			\label{e:3zU}\]
			i. e., \ref{prt:uglydist} is false with $z=\f{\phi}_h(v)$,
			\item \label{prt:extbadsecond} if for some $h' \in \mc{K}$ and $y,x' \in V$ the sequence $(h,h',v,x',y)$ has Properties \ref{prt:badend} and \ref{prt:badimage} of a bad situation then 
			\[ \tag{z.4.B} \label{e:4zB}\text{ if } y \in \dom(\phi_h) \text{ and } x',y \in \dom(\phi_{h'})\text{  holds then }\]
			\[z R f(\phi_{h}(y)) \iff \phi_{h'}(x') R f(\phi_{h'}(y)),\]
			i. e., again, \ref{prt:baddist} is false with $z=\f{\phi}_h(v)$, 
			\[ \tag{z.4.U}
			\label{e:4zU} \text{ if } y \in \dom(\phi_h),y=x'   \text{ and } y \not \in \dom(\phi_{h'}) \text{ holds then } z \lnot R f(\phi_{h}(y)), \]
			i. e., \ref{prt:uglydist} is false with $z=\f{\phi}_h(v)$.
			
		\end{enumerate}
		\begin{claim}
			There exists such a $z$.
		\end{claim}
		\begin{proof}[Proof of the Claim.]
			Since $f$ has property $(*)_0$ it is enough to show that requirements on $z$ \ref{prt:extauto}-\ref{prt:extbadsecond} are not contradicting. Obviously, by the injectivity of $\phi_h$ there is no contradiction between requirements of type \ref{prt:extauto}, and by the fact that only non-relations are required between requirements of type \ref{e:3zU} and of type \ref{e:4zU}. Thus, it is enough to check that requirements of type \ref{e:3zB}, \ref{e:3zB}- \ref{e:3zU}, \ref{e:4zB}, \ref{e:4zB}- \ref{e:4zU}, and of type \ref{prt:extauto}-\ref{prt:extbadfirst}, \ref{prt:extauto}-\ref{prt:extbadsecond} and \ref{prt:extbadfirst}-\ref{prt:extbadsecond} are not in a contradiction.
			
			\textit{The requirements in \ref{prt:extbadfirst} are compatible, \eqref{e:3zB}.} Suppose otherwise, namely, there is a contradiction between requirements of type \eqref{e:3zB}. Then we have automorphisms $h'_1,h'_2 \in \mc{K}$, vertices $x_1,x_2,x'_1,x'_2$ showing the contradiction, that is, $(h,h'_1,x_1,x'_1,v)$ has properties \ref{prt:badend}, \ref{prt:badimage} of a bad situation and $x_1 \in \dom(\phi_h) \text{ and } x'_1,v \in \dom(\phi_{h'_1})$ and similarly for $(h,h'_2,x_2,x'_2,v)$ but 
			\begin{equation}
			\label{e:equiv1}
			\phi_{h'_1}(x'_1) R f(\phi_{h'_1}(v)) \text{ and } \phi_{h'_2}(x'_2) \lnot R f(\phi_{h'_2}(v))
			\end{equation}
			and $f^{-1}(\phi_h(x_1))=f^{-1}(\phi_h(x_2))$, or equivalently, $x_1=x_2$. We claim that there exists an $(h'_1,h'_2,x'_1,x'_2,v)$ bad situation which contradicts the fact that  $(g, (\phi_h)_{h \in \mc{K}},M)$ was a good triple:
			\begin{enumerate}
				\item[\ref{prt:badend}] $(h,h'_1,x_1,x'_1,v)$ and $(h,h'_2,x_2,x'_2,v)$ have property \ref{prt:badend}, in particular $v=e(v,h'_1 \circ g)=e(v,h'_2 \circ g)$ or $v \in N$ and $x'_1=b(x'_1,h'_1 \circ g)$ or $x'_1 \in N$ and $x'_2=b(x'_2,h'_2 \circ g)$ or $x'_2 \in N$,
				\item[\ref{prt:badimage}] using \ref{prt:badimage} for $(h,h'_1,x_1,x'_1,v)$ and $(h,h'_2,x_2,x'_2,v)$ we get $h^{-1}(x_1)=h'^{-1}_1(x'_1)$ and $h^{-1}(x_2)=h'^{-1}_2(x'_2)$, and using $x_1=x_2$ we obtain $h'^{-1}_1(x'_1)=h'^{-1}_2(x'_2)$,
				\item[\ref{prt:baddist}] \eqref{e:equiv1} shows that this holds.
			\end{enumerate}
			
			\textit{The requirements in \ref{prt:extbadfirst} are compatible, \eqref{e:3zB} and \eqref{e:3zU}.} Suppose that there is a contradiction between requirements of type \eqref{e:3zB} and \eqref{e:3zU}. Then we have automorphisms $h'_1,h'_2 \in \mc{K}$, vertices $x_1,x'_1,x_2$ so that  $(h,h'_1,x_1,x'_1,v)$ and $(h,h'_2,x_2,v,v)$ have properties \ref{prt:badend} and \ref{prt:badimage}, $x_1,x_2 \in \dom(\phi_h), x'_1,v \in \dom(\phi_{h'_1})$, $v \not \in \dom(\phi_{h'_2})$ and
			\begin{equation}
			\label{e:equiv3}
			\phi_{h'_1}(x'_1) R f(\phi_{h'_1}(v))
			\end{equation}
			and $f^{-1}(\phi_h(x_1))=f^{-1}(\phi_h(x_2))$, that is, $x_1=x_2$. We claim that we have an $(h'_1,h'_2,x'_1,v)$ ugly situation: 
			
			\begin{enumerate}
				\item[\ref{prt:badend}] follows from the fact that $(h,h'_1,x_1,x'_1,v)$ and $(h,h'_2,x_2,v,v)$ have Property \ref{prt:badend}
				\item[\ref{prt:badimage}] similarly, $(h,h'_1,x_1,x'_1,v)$ and $(h,h'_2,x_2,v,v)$ have Property \ref{prt:badimage} of bad situations and by $x_1=x_2$ we get $h^{-1}(x_1)=h'^{-1}_1(x'_1)$ and $h^{-1}(x_2)=h'^{-1}_2(v)=h^{-1}(x_1)$, thus, \ref{prt:badimage} of bad situations holds for $(h'_1,h'_2,x'_1,v,v)$,
				\item[\ref{prt:uglynew}] since $(h,h'_2,x_2,v,v)$ has property \ref{prt:uglynew}, so $v \not \in \dom(\phi_{h'_2})$,
				\item[\ref{prt:uglydist}] finally, \eqref{e:equiv3} shows that this holds as well. 
			\end{enumerate}
			
			\textit{The requirements in \ref{prt:extbadsecond} are compatible, \eqref{e:4zB}.} Suppose that there is a contradiction between requirements of type \eqref{e:4zB}. Then we have automorphisms $h'_1,h'_2 \in \mc{K}$, vertices $y_1,x'_1,y_2,x'_2$ so that  $(h,h'_1,v,x'_1,y_1)$ and $(h,h'_2,v,x'_2,y_2)$ have properties \ref{prt:badend}, \ref{prt:badimage} and $y_1,y_2 \in \dom(\phi_h), x'_1,y'_1 \in \dom(\phi_{h'_1}) \text{ and } x'_2,y'_2 \in \dom(\phi_{h'_2})$ but
			\begin{equation}
			\label{e:equiv2}
			\phi_{h'_1}(x'_1) R f(\phi_{h'_1}(y_1)) \text{ and } \phi_{h'_2}(x'_2)  \lnot R f(\phi_{h'_2}(y_2))
			\end{equation}
			and $f(\phi_h(y_1))=f(\phi_h(y_2))$, that is, $y_1=y_2$. Then we have an $(h'_1,h'_2,x'_1,x'_2,y_1)$ bad situation:

			\begin{enumerate}
				\item[\ref{prt:badend}] follows from the fact that $(h,h'_1,v,x'_1,y_1)$ and $(h,h'_2,v,x'_2,y_2)$ have property \ref{prt:badend} and $y_1=y_2$,
				\item[\ref{prt:badimage}] $(h,h'_1,v,x'_1,y_1)$ and $(h,h'_2,v,x'_2,y_2)$ have property \ref{prt:badimage} so $h^{-1}(v)=h'^{-1}_1(x'_1)=h'^{-1}_2(x'_2)$  so this is also true,
				\item[\ref{prt:baddist}] \eqref{e:equiv2} shows that this property holds.
			\end{enumerate}

			\textit{The requirements in \ref{prt:extbadsecond} are compatible, \eqref{e:4zB} and \eqref{e:4zU}.} Suppose that there is a contradiction between requirements of type \eqref{e:4zB} and \eqref{e:4zU}. Then we have automorphisms $h'_1,h'_2 \in \mc{K}$, vertices $y_1,x'_1,y_2$ so that  $(h,h'_1,v,x'_1,y_1)$ and $(h,h'_2,v,y_2,y_2)$ have properties \ref{prt:badend}, \ref{prt:badimage} and $y_1,y_2 \in \dom(\phi_h)$, $x'_1,y_1 \in \dom(\phi_{h'_1})$, $y_2 \not \in \dom(\phi_{h'_2})$ but
			\begin{equation}
			\label{e:equiv4}
			\phi_{h'_1}(x'_1) R f(\phi_{h'_1}(y_1))
			\end{equation}
			and $f(\phi_h(y_1))=f(\phi_h(y_2))$, that is, $y_1=y_2$. We claim that we have an $(h'_1,h'_2,x'_1,y_1)$ ugly situation: 
			
			\begin{itemize}
				\item[\ref{prt:badend}] follows from the fact that $(h,h'_1,v,x'_1,y_1)$ and $(h,h'_2,v,y_2,y_2)$ have Property \ref{prt:badend} of bad situations and $y_1=y_2$,
				\item[\ref{prt:badimage}] similarly, $(h,h'_1,v,x'_1,y_1)$ and $(h,h'_2,v,y_2,y_2)$ have Property \ref{prt:badimage} of bad situations we get $h^{-1}(v)=h'^{-1}_1(x_1)$ and $h^{-1}(v)=h'^{-1}_2(y_2)=h'^{-1}_2(y_1)$,
				\item[\ref{prt:uglynew}] since $(h,h'_2,v,y_2,y_2)$ has property \ref{prt:uglynew}, so $y_2 \not \in \dom(\phi_{h'_2})$ and $y_1=y_2$,
				\item[\ref{prt:uglydist}] finally, \eqref{e:equiv4} shows that this condition is true as well.
			\end{itemize}
			
			\textit{The requirements in \ref{prt:extauto} and \ref{prt:extbadfirst} are compatible.} Otherwise there would be vertices $x,x'$ satisfying Property \ref{prt:badend} from Definition \ref{d:badsitu} and $w \in \dom(\phi_h)$  so that $f^{-1}(\phi_h(x))=\phi_h(w)$. 
			
			Suppose first $x \not \in N$. Then clearly $\mc{O}^f(\phi_h(x))=\mc{O}^f(\phi_h(w))$ and $\phi_h(x),\phi_h(w) \not \in N$. Using that $(g, (\phi_h)_{h \in \mc{K}},M)$ is a good triple by Property \ref{prt:distorb} we obtain $\mc{O}^{h \circ g}(x)=\mc{O}^{h \circ g}(w)$ and by $\phi_h(x) \not \in N$ the orbit $\mc{O}^f(\phi_h(x))$ is infinite. Since by Property \ref{prt:badend} of a bad situation $x=b(x,h \circ g)$ we get that $(h \circ g)^k(x)=w$ for some $k \geq 0$. Thus, by Properties \ref{prt:dom} and \ref{prt:conj} of good triples we get
			\[\phi_h(w)=\phi_h((h \circ g)^{k}(x))=f(\phi_h((h \circ g)^{k-1}(x)))=\dots=f^k(\phi_h(x)).\]
			But, this together with $f^{-1}(\phi_h(x))=\phi_h(w)$ contradicts the fact that $\mc{O}^f(\phi_h(x))$ is infinite.
			
			Now if $x \in N$ then clearly $\phi_h(x)=x$ so $f^{-1}(x)=\phi_h(w)$ is also is an element of $N$ by the fact that $N$ is the union of orbits of $f$, and therefore $\phi_h(w)=w$ by Property \ref{prt:nnoying} of good triples. Moreover, by \ref{prt:badimage} we have $x'=h'(h^{-1}(x))$, so $h|_N=h'|_N=f|_N$ implies $x' \in N$ and $x=x'$. Thus, since the requirements are contradicting by our assumption we get \[x \lnot R f(\phi_{h'}(v)) \iff vRw \iff \phi_{h'}(v)Rw\]
			but $w=f^{-1}(x)$ which is impossible.

			\textit{The requirements in \ref{prt:extauto} and \ref{prt:extbadsecond} are compatible.} The argument here is similar. Otherwise there would be a vertex $y$ satisfying Property \ref{prt:badend} from Definition \ref{d:badsitu} and $w \in \dom(\phi_h)$  so that $f(\phi_h(y))=\phi_h(w)$. 
			
			Suppose $y \not \in N$. Then clearly $\mc{O}^f(\phi_h(y))=\mc{O}^f(\phi_h(w))$, so $\phi_h(y),\phi_h(w) \not \in N$ and $\mc{O}^f(\phi_h(y)) \cap N = \emptyset$, thus $\mc{O}^f(\phi_h(y))$ is infinite. Again, by Property \ref{prt:distorb} we obtain $\mc{O}^{h \circ g}(y)=\mc{O}^{h \circ g}(w)$. Since by Property \ref{prt:badend} of a bad situation $y=e(y,h \circ g)$ we get that $(h \circ g)^k(w)=y$ for some $k \geq 0$. But this, using Property \ref{prt:conj} contradicts $f(\phi_h(y))=\phi_h(w)$ and the fact that the orbit $\mc{O}^f(\phi_h(y))$ is infinite.
			
			Now if $y \in N$ then clearly $\phi_h(y)=\phi_{h'}(y)=y$ so $f(y)=\phi_h(w)$ is also an element of $N$ thus $\phi_h(w)=w$. Our requirements are contradicting, so			
			\[\phi_{h'}(x') Rf(\phi_{h'}(y)) \iff v\lnot Rw.\]			
			But $y,f(y), w \in N$ so \[\phi_{h'}(x') Rf(\phi_{h'}(y)) \iff f(y) R \phi_{h'}(x') \iff \phi_{h'}(f(y))R \phi_{h'}(x') \iff f(y) R x'\] 
			using that $x'=h'(h^{-1}(v))$ and $f|_N=h|_N=h'|_N$ 
			\[\iff f(y) R h'(h^{-1}(v)) \iff h(h'^{-1}(f(y))) R v \iff f(y) R v.\]
			But $w=f(y)$, so this gives \[f(\phi_{h'}(y)) R \phi_{h'}(x') \iff wRv,\]
			showing that this is impossible.

			\textit{The requirements in \ref{prt:extbadfirst} and \ref{prt:extbadsecond} are compatible.} Suppose not, then we have sequences $(h,h'_1,x_1,x'_1,v)$ and $(h,h'_2,v,x'_2,y_2)$ having properties \ref{prt:badend} and \ref{prt:badimage}, $x_1,y_2 \in \dom(\phi_h)$ with 
			$f^{-1}(\phi_h(x_1))=f(\phi_{h}(y_2))$. Then $\mc{O}^f(\phi_h(x_1))=\mc{O}^f(\phi_h(y_2))$. 
			
			Let $x_1 \not \in N$, then $\phi_h(x_1),\phi_h(y_2) \not \in N$. Again, by Property \ref{prt:distorb} we obtain $\mc{O}^{h \circ g}(x_1)=\mc{O}^{h \circ g}(y_2)$. But by $y_2=e(y_2,h \circ g)$ we get that $(h \circ g)^k(x_1)=y_2$ for some $k \geq 0$. But this contradicts $f^2(\phi_h(y_2))=\phi_h(x_1)$.
			
			Now, if $x_1 \in N$ then $y_2 \in N$ holds as well. 
			Then, as we have seen before $\phi_h(y_2)=\phi_{h'_2}(y_2)=y_2$ and also $x'_1=x_1$.
			Again, by $h|_N=h'_2|_N=f|_N$ we get
			\[f(y_2)  R \phi_{h'}(x'_2) \iff f(y_2) R x'_2 \iff\]\[ f(y_2) R h'_2(h^{-1}(v)) \iff h(h'^{-1}_2(f(y_2)))  R v \iff f(y_2)R v.\]
			Moreover, using again $x'_1=x_1 \in N$, $f^{-1}(x_1) \in N$, $h|_N=h'_1|_N=f|_N$ and $\phi_{h'_1}|_N=id_N$ we obtain
			\[\phi_{h'_1}(x'_1) R f(\phi_{h'_1}(v)) \iff f^{-1}(x_1) R \phi_{h'_1}(v) \iff f^{-1}(x_1)Rv,\]
			so recalling that $f^2(y_2)=x_1$ we can conclude that the requirements are not in a contradiction.

		\end{proof}
		
		We return to the proof of Lemma \ref{l:extend}. Extend $\phi_h$ to $v$ defining  $\f{\phi}_h=\phi_h \cup \langle v,z\rangle$ for some $z$ having properties \ref{prt:extfirst}-\ref{prt:extbadsecond}. We check that $(g, (\phi_{h'})_{h' \in \mathcal{K}, h'|_{M^*} \not =h|_{M^*}} \cup (\f{\phi}_{h'})_{h \in \mathcal{K}, h'|_{M^*} =h|_{M^*}},M)$ is still a good triple, going through the definition of good triples. 
		
		\begin{enumerate}[label=(\roman*)]
			\item We have to check that $\f{\phi}_h$ is a partial automorphism, but this is exactly property \ref{prt:extauto} of $z$.
			
			\item Obvious, as $\f{\phi}_h$ is the extension of $\phi_h$ to a point $v$ already in $M$.
			\item Obvious.
			\item If  $(\f{\phi}_{h} \circ h \circ g)(v_0)= (f \circ \f{\phi}_{h})(v_0)$, became false after the extension for some $v_0$, then either $(h \circ g)(v_0)=v$ or $v_0=v$. We can exclude both of the possibilities, as by Property \ref{prt:dom} of the good triples $\dom(\phi_h) \supset rd(h \circ g)$, so $\phi_h$ would have been already defined on $v$.  
			
			\item If for $w,w' \in \dom(\f{\phi}_h) \setminus N$ we have that $\mc{O}^{h \circ g}(w) \not = \mc{O}^{h \circ g}(w')$ then clearly $w \not = w'$. Also, either $w,w' \in \dom(\phi_h)$, in which case we are done, or say, $w=v$. But by property \ref{prt:extfirst} of $z$ we have $\mc{O}^{f}(z) \cap  \mc{O}^{f}(\phi_h(w'))= \emptyset$ and clearly $\mc{O}^{f}(\f{\phi}_h(v))= \mc{O}^{f}(z)$  .
			\item Obvious, since we defined the extension of $\phi_h$ the same for a set of $h \in \mathcal{K}$ with the same restriction to $M^*$.
			
			\item This property does not use the functions $\phi_h$.
			
			\item By property \ref{prt:extfirst} of $z$ we have $z \lnot R f(z)$, so whenever $f(\f{\phi}_h(w)) R \f{\phi}_h(w)$ then clearly $w \not =v$, so $w \in \dom(\phi_h)$ and $(g, (\phi_h)_{h \in \mc{K}},M)$ was a good triple, thus,  $(g, (\f{\phi}_h)_{h \in \mc{K}},M)$ also has Property \ref{prt:onecase}.
			\item If there exists an $(h_1,h'_1,\f{\phi}_{h_1},\f{\phi}_{h'_1},g,x,y)$ ugly situation then one of $h_1, h'_1$ must coincide with $h$ on $M^*$, so as the definition of ugly situation depends only on $h|_{M^*}$, we can suppose that one of the functions is $h$. 
			
			Note that $h_1|_{M^*}=h'_1|_{M^*}$ would imply $\f{\phi}_{h}=\f{\phi}_{h'}$ contradicting \ref{prt:uglynew} and \ref{prt:uglydist}. Hence we can suppose that $\{h_1,h'_1\}=\{h,h'\}$ for some $h'|_{M^*} \not = h|_{M^*}$. Moreover, if $h'_1=h$ or $x,y \in \dom(\phi_h)$ by $\f{\phi_{h'}}=\phi_{h'}$ and \ref{prt:uglynew} we would already have an $(h',h,\phi_{h'},\phi_h,g,x,y)$ or $(h,h',\phi_{h},\phi_{h'},g,x,y)$ ugly situation, which is impossible as we have started with a good triple. 
			
			Thus, $h_1=h$ and $x=v$ or $y=v$ and by the definition of the ugly situation $(h,h',\f{\phi}_h,\phi_{h'},x,y,y)$ has Properties 
			\ref{prt:badend}, \ref{prt:badimage} and \ref{prt:uglynew} and  
			\[\f{\phi}_{h}(x) R f(\f{\phi}_{h}(y))\tag{*U}.\]
			
			Now suppose that $x \in \dom(\f{\phi}_h) \setminus \dom(\phi_h)$, that is, $x=v$ and $y \in \dom(\phi_h)$. Then since $(h,h',\phi_h,\phi_{h'},g,x,y,y)$ has Properties \ref{prt:badend}, \ref{prt:badimage} and \ref{prt:uglynew} the requirement \eqref{e:4zU} on $z$ ensures that 
			\[z \lnot R f(\phi_h(y)) \text{ or, equivalently } \f{\phi}_h(x) \lnot R f(\phi_h(y)),\]
			a contradicting (*U).
			
			Suppose $y \in \dom(\f{\phi}_h) \setminus \dom(\phi_h)$, $y=v$ and $x \in \dom(\phi_h)$. Then again, $(h,h',\phi_h,\phi_{h'},g,x,y,y)$ has \ref{prt:badend}, \ref{prt:badimage} and \ref{prt:uglynew}, now we use the requirement \eqref{e:3zU} on $z$: 
			
			\[z \lnot R f^{-1}(\phi_h(x)) \text{ or, equivalently by $v=y$} \]\[\f{\phi}_h(v) \lnot R f^{-1}(\phi_h(x)) \iff f(\f{\phi}_h(y)) \lnot R \f{\phi}_h(x),\]
			a contradiction.
			
			Finally, if $x,y \in \dom(\f{\phi}_h) \setminus \dom(\phi_h)$, that is $x=y=v$. Then we obtain that $\f{\phi}_{h}(x) R f(\f{\phi}_{h}(y))$ means $z R f(z)$, but this contradicts Property \ref{prt:extfirst} of $z$.

			\item 
			
			Suppose that there exists an $(h_1,h'_1,\f{\phi}_{h_1},\f{\phi}_{h'_1},g,x,x',y)$ bad situation. Again, we can suppose that at least one of $h_1$ and $h'_1$ equals to $h$ on $M^*$ and by symmetry there exists an $(h,h',\f{\phi}_{h},\f{\phi}_{h'},x,x',y)$ bad situation. 
			
			Note that using $x,x' \in M$ and $h^{-1}(x),h'^{-1}(x') \in \mc{K}^{-1}(M) \subset M^*$ we obtain that $h_1|_{M^*}=h'_1|_{M^*}$ would imply $x=x'$ and $\f{\phi}_{h_1}=\f{\phi}_{h'_1}$ which contradict \ref{prt:baddistb}. Thus, $h'|_{M^*} \not = h|_{M^*}$ so $\f{\phi}_{h'}=\phi_{h'}$.

			Clearly, at least one of vertices $x,x',y$ must be equal to $v$, hence otherwise there would be an $(h,h',\phi_{h},\phi_{h'},g)$ bad situation.
			
			Suppose that $y=v$ and $x \not =v$. Then by requirement \eqref{e:3zB} on $z$ and $z=\f{\phi}_h(v)$ we obtain 
			\[f(\f{\phi}_h(v)) R \phi_h(x) \iff \phi_{h'}(x') R f(\phi_{h'}(v)),\]
			showing that this is impossible.
			
			Suppose now $x=v$ and $y \not = v$. Then by requirement \eqref{e:4zB} on $z$ we get 
			\[z R f(\phi_{h}(y)) \iff \phi_{h'}(x') R f(\phi_{h'}(y)),\]
			or reformulating the statement
			\[\f{\phi}_f(v) R f(\phi_{h}(y)) \iff \phi_{h'}(x') R f(\phi_{h'}(y)),\]
			again, showing that $(h,h',\f{\phi}_{h},\phi_{h'},g,x,x',y)$ is not a bad situation.
			
			Finally, if $x=y=v$, property \ref{prt:baddist} would give that 
			\[z R f(z) \iff \phi_{h'}(x') R f(\phi_{h'}(v)),\]
			is not true. By Property \ref{prt:extfirst} of $z$ we get $z \lnot R f(z)$ so \begin{equation}
			\phi_{h'}(x') R f(\phi_{h'}(v)). \label{e:equivlast}                                                                         
			\end{equation}
			
			Now we claim that there is an $(h',h,\phi_{h'},\phi_h,g,x',v)$ ugly situation:
			\begin{itemize}
				\item[\ref{prt:badend}] follows from the facts that $(h,h',\f{\phi}_{h},\phi_{h'},g,x,x',y)$ is a bad situation, $x=v$ and $y=v$,
				\item[\ref{prt:badimage}] again, as $(h,h',\f{\phi}_{h},\phi_{h'},g,x,x',y)$ is a bad situation, we have $h^{-1}(x)=h'^{-1}(x')$, so by $x=v$ we have $h'^{-1}(x')=h^{-1}(v)$ which shows this property,
				\item[\ref{prt:uglynew}] clear since $v \not \in \dom(\phi_h)$,
				\item[\ref{prt:uglydist}] \eqref{e:equivlast} is exactly what is required.
				
			\end{itemize}
			This contradicts the fact that $(g, (\phi_h)_{h \in \mc{K}},M)$ was a good triple.
		\end{enumerate}
	\end{proof}
	\begin{corollary}
		\label{c:extend}
		Suppose that $v$ is a vertex, $(g, (\phi_h)_{h \in \mc{K}},M)$ is a good triple and $v \in M$. Then there exist extensions $\f{\phi}_{h} \supset \phi_{h}$ so that $(g, (\f{\phi}_h)_{h \in \mc{K}},\f{M})$ is a good triple and $v \in \bigcap_{h \in \mc{K}} \dom(\f{\phi}_h)$.
	\end{corollary}
	\begin{proof}
		First notice that by the compactness of $\mc{K}$ the set $\{h|_{M^*}:h \in \mathcal{K}, v \not \in \dom(\phi_h)\}$ is finite. By Lemma \ref{l:extend} we can define the extensions one-by-one for every element of $\{h|_{M^*}:h \in \mathcal{K}, v \not \in \dom(\phi_h)\}$.
	\end{proof}

	Finally, before we prove our main result we need a lemma about backward extension of the functions $\phi_h$.
	\begin{lemma}
		\label{l:extspec}
		Suppose that $(g, (\phi_h)_{h \in \mc{K}},M)$ is a good triple, $h \in \mc{K}$ and $z$ a vertex. Then for every $h \in \mc{K}$ there exists extensions $\f{\phi}_{h} \supset \phi_h$ and $\f{M} \supset M$ so that $(g, (\f{\phi}_h)_{h \in \mc{K}},\f{M})$ is a good triple and $z \in \bigcap_{h \in \mc{K}}\mc{O}^f(\ran(\f{\phi}_h))$.
	\end{lemma}
	
	\begin{proof}
		Clearly, the set  $\{h|_{M^*}:h \in \mc{K}, z \not \in \mc{O}^f(\ran(\phi_h))\}$ is finite. Let $\tau_{h|_{M^*}}:M^* \to 2$ so that 
		\begin{equation}
		\label{e:automor}
		\tau_{h|_{M^*}}(w)=0 \iff \phi_{h|_{M^*}}(w) \lnot R z
		\end{equation}
		
		and define $\tau_{h|_{M^*}}$ on ${M^*} \setminus \dom(\phi_{h|_{M^*}})$ arbitrarily.

		We claim that there exists a finite set of vertices $\{v_{h|_{M^*}}:h \in \mc{K}, z \not \in \ran(\phi_h)\}$ which are splitting points for $M^*$ and $\mc{K}$, $v_{h|_{M^*}}$ realizes $\tau_{h|_{M^*}}$ and 
		\begin{equation}
		\label{e:distance}                                                                                                                                                  
		d_{\mc{K}}(v_{{h|_{M^*}}}, M \cup \{v_{h'|_{M^*}}:h'|_{M^*} \not =h|_{M^*}\})>2:                                                                                                                                                                                                \end{equation}
		in order to see this, by the fact that the set $\{h|_{M^*}:h \in \mc{K}\}$ is finite, we can enumerate it as $\{p_0,\dots,p_k\}$. Now by Corollary \ref{c:splitting} we can choose inductively for every $i \leq k$ a $v_{p_i}$ splitting point for $M^*$ and $\mc{K}$ so that $d(v_{p_i},M^* \cup \{v_{p_j}:j<i\})>2$. Now, if $h$ is given then $h|_{M^*}=p_i$ for some $i$. If $d(v_{p_i},M^* \cup \{v_{h'|_{M^*}}:h'|_{M^*} \not =h|_{M^*}\}) \leq 2$ then since $d(v_{p_i},M^*)>2$ there was an $i' \not =i$ so that $d(v_{p_{i'}},v_{p_i}) \leq 2$. But this is impossible by $d(v_{p_i},\{v_{p_j}:j<i\})>2$. 
		
		Let $\f{\phi}_{h}=\phi_h \cup \langle v_{h|_{M^*}},z\rangle$ if $z \not \in \mc{O}^f(\ran(\phi_h))$ and $\f{\phi}_h=\phi_h$ otherwise. Let $\f{M}=M \cup \{v_{h|_{M^*}}:h \in \mc{K}\}$. In order to prove the lemma it is enough to show that $(g, (\f{\phi}_h)_{h \in \mc{K}},\f{M})$ is a good triple. Note that by \ref{prt:forth} of good triples we have that $h|_{M^*}=h'|_{M^*}$ implies $\phi_h=\phi_h'$, but by the definition $v_{h|_{M^*}}$'s we also have that $h|_{M^*}=h'|_{M^*}$ implies $\f{\phi}_h=\f{\phi}_h'$.

		\begin{enumerate}[label=(\roman*)]
			
			\item For $h \in \mc{K}$ we check that the extension is still an automorphism, but for every $w \in \dom(\phi_{h})$ we have by \eqref{e:automor}
			\[wRv_{h|_{M^*}} \iff \tau_{h|_{M^*}}(w)=1 \iff\]\[  \phi_{h|_{M^*}}(w)  R z \iff 
			\f{\phi}_{h|_{M^*}}(w) R \f{\phi}_{h|_{M^*}}(v_{h|_{M^*}}).\]
			
			\item Clearly, $\bigcup_{h \in \mc{K}}\dom(\f{\phi}_h)\subset\bigcup_{h \in \mc{K}}\dom(\phi_h) \cup \{v_{h|_{M^*}}:h \in \mc{K}\} \subset \f{M}$.
			
			\item Obvious.
			\item If $(\f{\phi}_{h} \circ h \circ g)(v_0)= (f \circ \f{\phi}_{h})(v_0)$, became false after the extension for some $v_0$, then either $(h \circ g)(v_0)=v_{h|_{M^*}}$ or $v_0=v_{h|_{M^*}}$. Both cases are impossible, as $v_0 \in \dom(g) \subset M$ and $d_{\mc{K}}((h \circ g)(v_0),M) \leq 1$ so they would imply  $d_{\mc{K}}(v_{h|_M},M) \leq 1$ which contradicts \eqref{e:distance}.
			
			\item Let $h \in \mc{K}$. If for $w,w' \in \dom(\f{\phi}_h) \setminus N$ we have that $\mc{O}^{h \circ g}(w) \not = \mc{O}^{h \circ g}(w')$ then clearly $w \not = w'$. Also, either $w,w' \in \dom(\phi_h)$, in which case we are done, or say, $w=v_{h|_M}$ and $w' \in \dom(\phi_h)$. But $z \not \in \mc{O}^f(\ran(\phi_h))$ by the definition of the functions $\f{\phi}_h$. Therefore, using $\mc{O}^{f}(\f{\phi}_h(v_{h|_{M^*}}))= \mc{O}^{f}(z)$ and $\mc{O}^{f}(z) \cap  \mc{O}^{f}(\phi_h(w'))= \emptyset$ we are done.
			
			\item As mentioned above, already $h|_{M^*}=h'|_{M^*}$ implies $\f{\phi}_h=\f{\phi}_h'$, let alone $h|_{\f{M}^*}=h'|_{\f{M}^*}$.
			
			\item This property does not use the functions $\phi_h$.
			
			\item Fix an $h \in \mc{K}$. By the fact that $v_{h|_{M^*}}$ was a splitting  point for $M$ and $\mc{K}$ we have that $h^{-1}(v_{h|_{M^*}})=h'^{-1}(v_{h|_{M^*}})$ implies $h|_{M^*}=h'|_{M^*}$. But then $v_{h'|_{M^*}}=v_{h|_{M^*}} \in \dom(\f{\phi}_{h'})$ as well, so this condition cannot be violated by $w=v_{h|_{M^*}}$, therefore, $w \in \dom(\phi_h)$. By the fact that $(g, (\phi_h)_{h \in \mc{K}},M)$ is a good triple clearly $w \in \dom(\phi_{h'}) \subset \dom(\f{\phi}_{h'})$.

			\item Suppose that there exists an $h, h' \in \mathcal{K}$ and vertices $x,y$ forming an $(h,h',\f{\phi}_h,\f{\phi}_{h'},g,x,y)$ ugly situation. Notice first that if $h|_{M^*}=h'|_{M^*}$ implies $\f{\phi}_h=\f{\phi}_{h'}$ and this contradicts the conjunction of \ref{prt:uglynew} and \ref{prt:uglydist}. 
			
			Therefore, we have $h|_{M^*} \not =h'|_{M^*}$. Then we claim that Lemma \ref{l:firsttriv} can be used for $\f{\phi}_h,\f{\phi}_{h'}$ and $v=v_{h|_{M^*}}$ and $v'=v_{h'|_{M^*}}$. Indeed, since $h|_{M^*}\not = h'|_{M^*}$ and $v=v_{h|_{M^*}}$ is a splitting points for $\mc{K}$ and $M^*$ clearly $h^{-1}(v) \not =h'^{-1}(v)$ \eqref{e:distance} shows that the other condition of Lemma \ref{l:firsttriv} holds as well. So there is no $(h,h',\f{\phi}_h,\f{\phi}_{h'},g,x,y)$ ugly situation.

			\item Let us consider an $(h,h',\f{\phi}_{h},\f{\phi}_{h'},g,x,x',y)$ bad situation. Again, if $h|_{M^*}=h'|_{M^*}$ then $\f{\phi}_h=\f{\phi}_{h'}$ and $x=x'$. But then \ref{prt:baddist} must fail. 
			
			So $h|_{M^*} \not = h'|_{M^*}$. Then again, the assumptions of Lemma \ref{l:firsttriv} hold for $\f{\phi}_h,\f{\phi}_{h'}$ and $v=v_{h|_{M^*}}$ and $v'=v_{h'|_{M^*}}$. Using part (\ref{lp:bad}) we obtain that either $x=v_{h|_{M^*}}$, $x'=v_{h'|_{M^*}}$ or $y=v_{h|_{M^*}}=v_{h'|_{M^*}}$. But from \ref{prt:badimage} we have $d(x,x')<2$, so $d(v_{h|_{M^*}},v_{h'|_{M^*}}) \leq 2$, so in both cases we are in a contradiction with \eqref{e:distance}.

		\end{enumerate}
	\end{proof}
	Now we are ready to prove the main theorem of this section.
	
	\begin{proof}[Proof of Theorem \ref{t:randommain1}]
		Choose a vertex from each orbit of $f$ and enumerate these vertices as $\{z_0,z_1,\dots\}$ and recall that we have fixed an enumeration of $V$, $\{v_0,v_1,\dots\}$.
		
		By Lemma \ref{l:ind0} the triple $(g_0,(\phi_{0,h})_{h \in \mathcal{K}},M_0)=(id_N,(id_N)_{h \in \mathcal{K}},N)$ is good. 
		
		Suppose that we have already defined a good triple $(g_i,(\phi_{i,h})_{h \in \mathcal{K}},M_i)$ for every $i \leq n$ with the following properties:
		
		\begin{enumerate}
			\item \label{prt:indincr} $M_0 \subset M_1 \subset \dots \subset M_n$, $g_0 \subset g_1 \subset \dots \subset g_n$ and $\forall h \in \mathcal{K}$ we have $\phi_{h,0} \subset \phi_{h,1} \subset \dots \subset \phi_{h,n}$,

			\item \label{prt:indforth} if $2k < n$ then  \[\{v_0,v_1,\dots v_k\} \subset \ran(g_{2k}) \cap \dom(g_{2k}),\]
			
			\item \label{prt:indback} if $2k+1 \leq n$ \[\{z_0,z_1,\dots z_k\} \subset \bigcap_{h \in \mathcal{K}} \mc{O}^f(\ran(\phi_{h,2k+1})),\]
			
		\end{enumerate}
		
		We do the inductive step for an even $n+1$. Choose the minimal index $k$ (which is by the inductive assumption is $\geq \frac{n-1}{2}$) so that $v_k \not \in \ran(g_{n}) \cap \dom(g_{n})$. 
		
		First, by Remark \ref{rm:extm} we can extend $M_n$ to $M'_n \supset \{v_k, h(v_k):h \in \mc{K}\}$ so that $(g_n,(\phi_{h,n})_{h \in \mathcal{K}},M_n)$ is still a good triple. By Corollary \ref{c:extend} there exists an extension $g'_{n} \supset g_{n}$, $\phi'_{h,n} \supset \phi_{h,n}$ and $M''_{n} \supset M'_{n}$ so that $\{v_k,h(v_k):h \in \mc{K}\} \subset \bigcap_{h \in \mc{K}} \dom(\phi'_{h,n})$ and the extended triple is still good.
		
		Second, by Lemma \ref{l:forward} applied firstly and Lemma \ref{l:backward} applied secondly we get extensions $g_{n+1} \supset g'_{n}$, $\phi_{h,n+1} \supset \phi'_{h,n}$ and $M_{n+1} \supset M''_{n}$ so that $(g_{n+1},(\phi_{h,n+1})_{h \in \mathcal{K}},M_{n+1})$ is a good triple and $v_k \in \ran(g_{n+1}) \cap \dom(g_{n+1})$. This extension obviously satisfies the inductive hypothesis.

		Now we do the inductive step for an odd $n+1$ as follows: choose the minimal index $k$ ($\geq \frac{n}{2}$) so that $z_k \not \in \bigcap_{h \in \mathcal{K}} \mc{O}^f(\ran(\phi_{h,n}))$. By Lemma \ref{l:extspec} there exist extensions $g_{n+1} \supset g_{n}$, $\phi_{h,n+1} \supset \phi_{h,n+1}$ and $M_{n+1} \supset M_{n}$ so that $z_k \in \bigcap_{h \in \mathcal{K}} \mc{O}^f(\ran(\phi_{h,n+1}))$. This triple satisfies the inductive assumptions as well.
		
		Thus the induction can be carried out. We claim that $g=\bigcup_{n} g_n$ and $\phi_{h}=\bigcup_{n} \phi_{h,n}$ are automorphisms of $\mc{R}$ and for every $h \in \mc{K}$ we have 
		\[\phi_h \circ h \circ g=f \circ \phi_h.\]
		Indeed, as $g$ and $\phi_{h}$ are increasing unions of partial automorphisms, they are partial automorphisms as well. Moreover by assumption (\ref{prt:indforth}) of the induction $V=\{v_0,v_1,\dots\} \subset \rd(g)$, thus $g \in \aut(\mc{R})$. By \ref{prt:dom} of good triples we have $\dom(g_n) \subset \dom(\phi_{h,n})$ so 
		\[V=\bigcup_{n \in \omega} \dom(g_n) \subset\bigcup_{n \in \omega} \dom(\phi_{h,n})=\dom(\phi_h). \]
		
		By \ref{prt:conj} we obtain
		\[\phi_h \circ h \circ g=f \circ \phi_h.\]
		We have seen that $g \in \aut(\mc{R})$, so $\ran(h \circ g)=V$, therefore from the above equality we get
		\[\ran(\phi_h)=f(\ran(\phi_h)),\]
		so the set $\ran(\phi_h)$ is $f$ invariant, consequently contains full orbits of $f$. But by assumption (\ref{prt:indback}) of the induction $\ran(\phi_h)$ intersects each $f$ orbit, so $\phi_h \in \aut(\mc{R})$ as well.
		
		The second part of the theorem is obvious, as $id_N=g_0 \subset g$ and for every $h \in \mc{K}$ also $id_N=\phi_{h,0} \subset \phi_h$.
	\end{proof}

	\subsection{Translation of compact sets, general case}
	Now we give a complete characterization of the non-Haar null conjugacy classes in $\aut(\mc{R})$. Interestingly enough, a variant of the following property has already been isolated by Truss \cite{truss1985group}.
	\begin{definition}
		Let $f \in \aut(\mc{R})$. We say that $f$ has property $(*)$ if
		\begin{itemize}
			\item $f$ has only finitely many finite orbits and infinitely many infinite orbits,
			\item  for every finite set $M \subset V$ and $\tau:M \to 2$ there exists a $v$ that realizes $\tau$ and $v \not \in \mc{O}^f(M)$.
		\end{itemize}

	\end{definition}
	
	\begin{theorem}
		\label{t:randommain2}
		Suppose that $f$ has property $(*)$. Then the conjugacy class of $f$ is compact biter. If $f$ has no finite orbits, then the conjugacy class of $f$ is compact catcher.	\end{theorem}
	
	Our strategy is to reduce this theorem to the special case that has been proven in Theorem \ref{t:randommain1}.
	\begin{claim}
		\label{cl:randomred}
		Suppose that $f$ has property $(*)$. Let $N$ be the union of the finite orbits of $f$ and $\tau:N \to 2$. Then either 
		\begin{enumerate}
			\item \label{prt:noedge} for every $\f{N} \supset N$ finite and $\f{\tau} \supset \tau$, $\f{\tau}:\f{N} \to 2$ there exists a vertex $v$ that realizes $\f{\tau}$, so that $v \not \in \mc{O}^f(\f{N})$ and $v\lnot Rg(v)$ or
			\item \label{prt:isedge} for every $\f{N} \supset N$ finite and $\f{\tau} \supset \tau$, $\f{\tau}:\f{N} \to 2$ there exists a vertex $v$ that realizes $\f{\tau}$, so that $v \not \in \mc{O}^f(\f{N})$ and $vRg(v)$.
			
		\end{enumerate}
		(The possibilities are \textit{not} mutually exclusive.)
		
	\end{claim}
	\begin{proof}
		Suppose that neither of these holds. In other words, there exist finite sets $\f{N}, \f{N}' \supset N$ and $\f{\tau}:\f{N} \to 2$,  $\f{\tau}':\f{N}' \to 2$ extending $\tau$ so that for every $v$ that realizes $\f{\tau}$ and $v \not \in \mc{O}^f(\f{N})$ we have $vRf(v)$ and $v \lnot R f(v)$ that realizes $\f{\tau}'$ and $v \not \in \mc{O}^f(\f{N}')$. 
		
		Notice that as $f$ is an automorphism the fact that for every $v$ that realizes $\f{\tau}$ and $v \not \in \mc{O}^f(\f{N})$ we have $vRf(v)$ implies that for every $k$ if $v$ realizes $\f{\tau} \circ f^{-k}$ and $v \not \in \mc{O}^f(f^{k}(\f{N}))$ then $vRf(v)$. 
		
		Let $M=\f{N} \setminus N$ and $n \in \omega$ so that the length of each orbit in $N$ divides $n$. As $f$ has only infinite orbits outside of $N$, for large enough $k$ we have $f^{kn}(M) \cap \f{N}'=\emptyset$. Moreover, by the condition on $n$ we have that $\f{\tau} \circ f^{-kn}$ coincides with $\tau$ on $N$. But then $\f{\tau} \circ f^{-kn} \cup \f{\tau}'$ is a function extending $\tau$. Since $f$ has property $(*)$ there exists a $v \not \in \mc{O}^f(f^{kn}(\f{N}) \cup \f{N}')$ which realizes $\f{\tau} \circ f^{-kn} \cup \f{\tau}'$. Then on the one hand $v$ realizes $\f{\tau} \circ f^{-kn}$ and $v \not \in \mc{O}^f(f^{kn}(\f{N}))$ so, as mentioned above, $vRf(v)$. On the other hand it also realizes $\f{\tau}'$ and $v \not \in \mc{O}^f(\f{N}')$ thus $v \lnot R f(v)$, a contradiction. 
	\end{proof}
	\begin{proof}[Proof of Theorem \ref{t:randommain2}.]
		Let $N$ be the union of the finite orbits of $f$. Define a function $\sigma: \{\tau:N \to 2\} \to 2$ as follows: let $\sigma(\tau)=0$ if condition (\ref{prt:noedge}) holds from Claim \ref{cl:randomred} and $\sigma(\tau)=1$ otherwise. Moreover, define an equivalence relation $\simeq$ on  $\{\tau:N \to 2\}$ by $\tau \simeq \tau'$ if there exists a $k \in \mathbb{Z}$ such that $\tau \circ f^{k}=\tau'$. Note that if $\tau \simeq \tau'$ then $\sigma(\tau)=\sigma(\tau')$: suppose that (\ref{prt:noedge}) holds for $\tau$ and $\tau'=\tau \circ f^{k}$ and let $\f{\tau}' \supset \tau'$. Then, as $\f{\tau}' \circ f^{-k} \supset \tau$, there exists a $v$ realizing $\f{\tau}' \circ f^{-k}$, $v \not \in \mc{O}^f(\dom(\f{\tau}' \circ f^{-k}))=\mc{O}^f(\dom(\f{\tau}'))$ and $v \lnot R f(v)$. But then $f^{-k}(v) \lnot R f^{-k+1}(v)$, $f^{-k}(v) \not \in \mc{O}^f(\dom(\f{\tau}'))$ and $f^{-k}(v)$ realizes $\f{\tau}'$. Thus, we can consider $\sigma$ as a $\{\tau:N \to 2\}/_\simeq \to 2$ map.
		
		Let \[V_{[\tau]}=\{v \in V \setminus N:v \text{ realizes some $\tau' \simeq \tau$}\}.\]
		Then clearly $V$ is the disjoint union of the sets $N$ and $V_{[\tau]}$ for $\simeq$ equivalence classes of maps $\tau:N \to 2$. The idea is to switch the edges and non-edges in every set  $V_{[\tau]}$ according to $\sigma$: let us define an edge relation $R'$ on the vertices $V$ as follows: for every distinct $v,w \in V$ if $v,w \in V_{[\tau]}$ for some $\tau$ and $\sigma([\tau])=1$ let $vR'w \iff v\lnot Rw$, otherwise let $vR'w \iff  vRw$. 
		
		\begin{claim}
			\label{cl:second}
			There exists an isomorphism $S:(V,R') \to (V,R)$ so that $S|_N=id|_N$ and for every $\tau$ we have $S(V_{[\tau]})=V_{[\tau]}$. Moreover, the subgroup $G_f=\{h \in \aut(\mc{R}):h|_N=f^{k}|_N \text{ for some } k \in \mathbb{Z}\}$ is invariant under conjugating with $S$ (we consider $S$ here as an element of $Sym(V)$, which is typically not an automorphism of $\mc{R}$) and for every $h \in G_f$ we have $h(N)=N$ and $h(V_{[\tau]})=V_{[\tau]}$ for each map $\tau:N \to 2$.
		\end{claim}
		\begin{proof}
			We define $S$ by induction, using a standard back-and-forth argument. Let us start with $S_0|_N=id|_N$ and suppose that we have already defined $S_n$ a partial isomorphism that respects the sets $V_{[\tau]}$ so that $\{v_0,v_1,\dots,v_n\} \subset  \ran(S_n) \cap \dom(S_n)$. Now we want to extend $S_n$ to $v_{n+1}$. Let $\tau$ be so that $v_{n+1} \in V_{[\tau]}$ and $v_{n+1}$ realizes $\tau$. Let us define $\rho: \ran(S_n) \to 2$ as $\rho(z)=1 \iff S^{-1}_{n}(z)Rv_{n+1}$. Clearly, in order to prove that $S_n$ can be extended it is enough to check that there exists a $z_{n+1} \in V_{[\tau]}$ realizing $\rho$ with respect to the relation $R'$. Let us define $\rho'$ as \[\rho'(z)=
			\begin{cases}
			1-\rho(z), &\text{ if } z \in V_{[\tau]}\\
			\rho(z), &\text{ otherwise.}
			\end{cases}
			\]
			Then, by property $(*)$ of $f$ there exists a vertex $z_{n+1}$ that realizes $\rho'$ with respect to $R$ and also $\rho' \supset \tau$ so $z_{n+1} \in V_{[\tau]}$. But by the definition of $R'$, as $R'$ was obtained by switching the edges within the sets $V_{[\tau]}$, clearly $z_{n+1}$ realizes $\rho$ with respect to $R'$. The ``back" part can be proved similarly.
			
			In order to prove the second claim suppose that $h|_N=f^k|_N$ for some $k$. It is clear that since $N$ is the union of orbits of $f$ it must be the case for $h$ as well, so $h(N)=N$. First, we claim that for every $\tau:N \to 2$ we have $h(V_{[\tau]})=V_{[\tau]}$: let $v \in V_{[\tau]}$ and $\tau' \simeq \tau$ so that $v$ realizes $\tau$. Then $h(v)$ realizes $\tau \circ h^{-1}$, but we have $h^{-1}|_N=f^{-k}|_N$ thus $h(v)$ realizes $\tau \circ f^{-k}$, so by definition $h(v) \in V_{[\tau]}$. 
			
			Now we check that $S^{-1}hS$ is an automorphism of $\mc{R}$. Take arbitrary vertices $x,y \in V$. If for some $\tau$ we have $x,y \in V_{[\tau]}$ then $xRy \iff x \lnot R' y$ and $xRy \iff S(x) R' S(y)$ and since $h$ and $S$ fix the sets $V_{[\tau]}$ we have $S(x),S(y) \in V_{[\tau]}$ and $h$ is an automorphism, \[xRy \iff S(x) R' S(y) \iff S(x) \lnot R S(y) \iff h(S(x)) \lnot R h(S(y)) \iff\]\[ h(S(x)) R' h(S(y)) \iff S^{-1}(h(S(x)))RS^{-1}(h(S(y))).\] If $x$ and $y$ are in different parts of the partition $V=N\cup \bigcup_{[\tau]} V_{[\tau]}$, then the statement is obvious, as in this case $R$ coincides with $R'$. 
			
		\end{proof} 
		Thus, conjugating with $S$ induces an automorphism $\f{S}$ of the group $G_f$. 
		\begin{claim}
			\label{cl:third}
			$\f{S}(f)$ has property $(*)_0$ from Definition \ref{d:star0}, $\f{S}(f)|_N=f|_N$ and $N$ is the union of finite orbits of $\f{S}(f)$. 
		\end{claim}
		\begin{proof}
			The second part of the claim is obvious: conjugating does not change the cardinality of orbits so $\f{S}(f)$ has infinitely many infinite orbits and finitely many finite ones and also $S|_N=id|_N$ so $S^{-1}fS|_N=f|_N$.
			
			Now take a finite set $M$ and a map $\tau:M \to 2$. Without loss of generality we can suppose $N \subset M$. Then, define
			$\rho:S(M) \to 2$ as follows: 
			\[\rho(w)=
			\begin{cases}
			1-\tau(S^{-1}(w)), &\text{ if } w \in V_{[\tau|_N]} \text{ and } \sigma(V_{[\tau|_N]})=1\\
			\tau(S^{-1}(w)), &\text{ otherwise.}
			\end{cases}
			\]
			Then there exists a $v_0 \not \in \mc{O}^f(S(M))$ so that $v_0$ realizes $\rho$ and 
			\begin{equation} \label{e:v0prop}
			v_0\lnot Rf(v_0) \text{ if } \sigma(V_{[\tau|_N]})=0 \text{ and } v_0Rf(v_0)\text{ if }\sigma(V_{[\tau|_N]})=1.
			\end{equation}
			Since $v_0$ realizes $\tau|_N$ and $\tau|_N \simeq \tau|_N \circ f^{-1}$ we have that $f(v_0)$ realizes $\tau|_N \circ f^{-1}$ thus $f(v_0) \in V_{[\tau|_N]}$. Let $v=S^{-1}(v_0)$, since $S$ fixes the sets $V_{[\tau|_N]}$ we have $v \in V_{[\tau|_N]}$ as well. 
			
			We show that $v$ realizes $\tau$. Let $w \in M$ be arbitrary. Suppose first that $\sigma(V_{[\tau|_N]})=0$ or $w \not \in V_{[\tau|_N]}$. Then from the fact that $v_0=S(v) \in V_{[\tau|_N]}$ we have
			\begin{equation}
			vRw \iff vR'w \iff S(v)RS(w) \iff v_0 R S(w) 
			\label{e:ending0}
			\end{equation}
			by the fact that $w \not \in V_{[\tau|_N]}$ or $\sigma(V_{[\tau|_N]})=0$
			\[ \iff\rho(S(w))=1  \iff \tau(S^{-1}(S(w)))=\tau(w)=1.\]
			
			Now, if $\sigma(V_{[\tau|_N]})=1$ and $w \in V_{\tau|_N}$ then from the definition of $\rho$ clearly  
			\begin{equation}
			vRw \iff v \lnot R'w \iff S(v) \lnot R S(w) \iff v_0 \lnot R S(w) \iff
			\label{e:ending1}
			\end{equation}
			\[ \rho(S(w))=0  \iff \tau(S^{-1}(S(w)))=\tau(w)=1.\]
			
			Moreover, using Claim \ref{cl:second} we get that $S$ and $f$ fixes the sets $V_{[\tau|_N]}$ and $v \in V_{[\tau|_N]}$ so clearly $(S^{-1} \circ f \circ S)(v) \in V_{[\tau|_N]}$. Thus, by equations \eqref{e:ending0} and \eqref{e:ending1} used for $w=(S^{-1} \circ f \circ S)(v)$ we obtain that $vR(S^{-1} \circ f \circ S)(v)$ is true if and only if either
			\[v_0RS((S^{-1} \circ f \circ S)(v) ) \text{ and } \sigma(V_{[\tau|_N]})=0\]
			or
			\[v_0\lnot RS((S^{-1} \circ f \circ S)(v)) \text{ and } \sigma(V_{[\tau|_N]})=1\]
			holds. 
			
			Now $v_0=S(v)$ so we get that $vR(S^{-1} \circ f \circ S)(v)$ holds if and only if either
			\[v_0RSf(v_0) \text{ and } \sigma(V_{[\tau|_N]})=0\]
			or
			\[v_0\lnot Rf(v_0) \text{ and } \sigma(V_{[\tau|_N]})=1\]
			holds. From this, using \eqref{e:v0prop} we get that $v\lnot R(S^{-1} \circ f \circ S)(v)$.
			
			Finally, we prove $v \not \in \mc{O}^{S^{-1}\circ f \circ S}(M)$. Suppose the contrary, let $w \in M$ so that $(S^{-1}fS)^k(w)=v$. Then $(S^{-1}fS)^k(w)=S^{-1}f^kS(w)$ so $S(v)=v_0 \in \mc{O}^{f}(S(M))$, contradicting the choice of $v_0$.
			
			Thus, $S^{-1}fS$ has property $(*)_0$.
		\end{proof}
		
		Now we are ready to finish the proof of the theorem. Let $\mc{K}_0 \subset \aut(\mc{R})$ be an arbitrary non-empty compact set. We will translate a non-empty portion of $\mc{K}_0$ into the conjugacy class of $f$. First, translating $\mc{K}_0$ we can suppose that there exists a non-empty portion $\mc{K}$ of $\mc{K}_0$ so that for every $h \in \mc{K}$ we have that $h|_N=f|_N$ (note that if $f$ has no finite orbits then $\mc{K}=\mathcal{K}_0$ is a suitable choice). In particular, $\mc{K} \subset G_f$. By Claim \ref{cl:second} $G_f$ is invariant under conjugating by $S$ and such a map is clearly an auto-homeomorphism of $G_f$, so $\f{S}(\mc{K})$ is also compact. Using Claim \ref{cl:third} $\f{S}(f)$ has property $(*)_0$ and we can apply the second part of Theorem \ref{t:randommain1} and we get a $g$ so that $g|_N=id|_N$ and for each $h \in \f{S}(\mc{K})$ an automorphism $\phi_h$ such that $\phi_h \circ h \circ g \circ \phi^{-1}_h=\f{S}(f)$ and $\phi_h|_N=id|_N$. In particular, all the automorphisms $g$ and $\phi_h$ are in $G_f$. We will show that $\f{S}^{-1}(g)$ translates $\mc{K}$ into the conjugacy class of $f$. Let $h \in \mc{K}\f{S}^{-1}(g)$ be arbitrary. Then of course $h=h'SgS^{-1}$ for some $h' \in \mc{K}$ and $S^{-1}hS=S^{-1}h'Sg$ so, as $S^{-1}h'S \in \f{S}(\mc{K})$ we get \[S^{-1}hS=\phi_{h'}^{-1}S^{-1}fS\phi_{h'}.\]
		Thus, \[h=S\phi^{-1}_{h'}S^{-1}fS\phi_{h'}S^{-1}\] and as $\phi_{h'} \in G_f$ and $G_f$ is $\f{S}$ invariant we have $S\phi_{h'}S^{-1} \in G_f \subset \aut(\mc{R})$. Therefore, $h$ is a conjugate of $f$ which finishes the proof.
		
	\end{proof}
	From Theorem \ref{t:randommain2} and Proposition \ref{p:F and C co-Haar null} we can deduce the complete characterization of the non-Haar null conjugacy classes of $\aut(\mc{R})$:
	
	\begin{theorem} For almost every element $f$ of $\aut(\mc{R})$
		\label{t:randomintro}
		\begin{enumerate}
			\item \label{pt:autr1} for every pair of finite disjoint sets, $A,B \subset V$ there exists $v \in V$ such that $(\forall x \in A)(xRv)$ and $(\forall y \in B)(y \nr v)$  \textit{ and $v \not \in \mathcal{O}^f(A \cup B)$, i. e., the union of orbits of the elements of $A \cup B$},
			\item \label{pt:autr2} (from Theorem \ref{t:gen}) $f$ has only finitely many finite orbits. 
		\end{enumerate}
		These properties characterize the non-Haar null conjugacy classes, i. e., a conjugacy class is non-Haar null if and only if one (or equivalently each) of its elements has properties \eqref{pt:autr1} and \eqref{pt:autr2}.
		
		Moreover, every non-Haar null conjugacy class is compact biter and those non-Haar null classes in which the elements have no finite orbits are compact catchers.  
	\end{theorem}

	\begin{proof}[Proof of Theorem \ref{t:randomintro}.] The facts that the classes of elements having properties  \ref{pt:autr1} and \ref{pt:autr2} and that these classes are compact biters (or catchers, when there are no finite orbits) is exactly Theorem \ref{t:randommain2}. 
		
		The only remaining thing is to show that the union of the conjugacy classes of elements not having properties \ref{pt:autr1} and \ref{pt:autr2} is Haar null. The collection of automorphisms having infinitely many finite orbits is Haar null by Theorem \ref{t:gen}.
		
		Now consider the set $\mc{C}_0=\{f \in \aut(\mc{R}): f $ has property \ref{pt:autr1}$\}$. Proposition \ref{p:F and C co-Haar null} states that the set $\mc{C}$ is co-Haar null for every $G$ having the $FACP$, in particular, for $\aut(\mc{R})$ the set
		\[\mathcal{C} = \{f \in \aut(\mc{R}) :\forall F \subset V \text{ finite } \forall 
		v \in V \;(\text{if $\aut(\mc{R})_{(F)}(v)$ is infinite} \]\[
		\text{then it is not covered by finitely many orbits of $f$})\}\]
		is co-Haar null. Thus, it is enough to show that $\mc{C}_0 \supset \mc{C}$ or equivalently $\aut(\mc{R})\setminus \mc{C}_0 \subset \aut(\mc{R})\setminus \mc{C}$. But this is obvious: if $f \not \in \mc{C}_0$ then there exist disjoint finite sets $A$ and $B$ such that the set $U=\{v:(\forall x \in A)(xRv)$ and $(\forall y \in B)(y \nr v)\}$ can be covered by the $f$ orbit of $A \cup B$. So, letting $F=A \cup B$ and noting that $U$ is infinite and  $\aut(\mc{R})_{(F)}$ acts transitively on $U \setminus F$ we get that for every $v \in U \setminus F$ the orbit $\aut(\mc{R})_{(F)}(v) \subset \mc{O}^f(F)$, showing that $f \not \in \mc{C}$. 
	\end{proof}

	\section{An application}
	\label{s:appl}

	Applying our results and methods about $\aut(\mc{R})$ one can prove a version of a theorem of Truss \cite{truss1985group}. Truss has shown first that if $f, g \in \aut(\mc{R})$ are non-identity elements then $f$ can be expressed as a product of five, later that it can be expressed as the product of three conjugates of $g$ \cite{truss2003automorphism}. Using the methods developed in Section \ref{s:autr} and the characterization of the non-Haar null classes of $\aut(\mc{R})$ one can prove this statement with four conjugates.
	
	\begin{theorem} 
		\label{t:truss}
		Let $C \subset Aut(\mc{R})$ be the conjugacy class of a non-identity element. Then $C^4 (=\{f_1f_2f_3f_4: f_1,f_2,f_3,f_4 \in C\})=\aut(\mc{R})$. 
	\end{theorem}
	The full proof of this theorem will be omitted, as this statement has already been known and writing down the new proof in detail would be comparable in length to the original proof. So, we split the proof into two propositions from which only the first one will be shown rigorously. 
	
	A certain conjugacy class plays an important role in the proof.	
	\begin{definition}
		\label{d:c0} Let $C_0$ be the collection of elements $f \in \aut(\mc{R})$ with the following properties
		
		\begin{enumerate}
			\item \label{prt:nofini} there are infinitely many infinite orbits and no finite ones,
			\item \label{prt:pattern} for every pair of finite disjoint sets, $A,B \subset V$ there exists $v \in V$ such that $v \not \in \mathcal{O}^f(A \cup B)$, $(\forall x \in A)(xRv)$, and $(\forall y \in \mathcal{O}^f(A \cup B) \setminus A)(y \nr v)$,  (in particular, $(\forall y \in B)(y \nr v)$),
			\item \label{prt:incycle} for every $v \in V$ and $k \in \mathbb{Z}$ we have $v \lnot R f^k(v)$,
			\item \label{prt:outcycle} for every $v,w$ the set $\{k\in \mathbb{Z} : v R f^k(w)\}$ is finite.
			
		\end{enumerate}
	\end{definition}
	
	Theorem \ref{t:truss} clearly follows from the following two propositions.
	
	\begin{proposition} 
		\label{p:c0prop}
		$C_0$ is a conjugacy class and $C^2_0=\aut(\mc{R})$.
	\end{proposition}
	
	\begin{proposition}
		\label{p:c0prop1}
		Let $C$ be the conjugacy class of a non-identity element. Then $C^2 \supset C_0$.
	\end{proposition}
	We will prove Proposition \ref{p:c0prop}, because it shows how our characterization can be used, and after that we only sketch the proof of Proposition \ref{p:c0prop1}. 
	
	\begin{proof}[Proof of Proposition \ref{p:c0prop}.] 
		Suppose that $f,f' \in C_0$. We first show that $f$ and $f'$ are conjugate by building an automorphism $\varphi$ so that $\varphi \circ f=f'\circ \varphi$. 
		
		Suppose that we have an $R$-preserving map $\varphi$ such that $\dom(\varphi)$ is the union of finitely many $f$ orbits, $\ran(\varphi)$ is the union of finitely many $f'$ orbits and $\varphi \circ f=f'\circ \varphi$ holds where both sides are defined. We extend $\dom(\varphi)$ and $\ran(\varphi)$ to every vertex back-and-forth. 
		
		Recall that $\{v_0,v_1\dots\}$ is an enumeration of the vertices of $\mc{R}$. Take the minimal $i$ with $v_i \not \in \dom(\varphi)$. Then, by condition (\ref{prt:outcycle}) on the map $f$, $v_i$ is only connected to  finitely many vertices from $\dom(\varphi)$, let us denote these vertices by $\{w_1,\dots,w_k\}$ and choose one element $\{w_{k+1},\dots,w_l\}$ from every $f$ orbit in the domain of $\varphi$ that is different from $\mc{O}^f(w_i)$ for every $i \leq k$.  
		
		Then, since condition (\ref{prt:pattern}) holds for $f'$ there exists a vertex $v'$ so that $v' \not \in \bigcup_{i \leq l} \mc{O}^{f'}(\varphi(w_i))$, $v' R \varphi(w_i)$ for $i \leq k$ and $v' \nr w$ whenever $w \in (\bigcup_{i \leq l} \mc{O}^{f'}(\varphi(w_i))) \setminus \{\varphi(w_i):i \leq k\}$. Let $\varphi(v_i)=v'$ and extend $\varphi$ to $\mc{O}^f(v_i)$ defining $\varphi(f^{n}(v_i))={f'}^{n}(\varphi(v_i))$. Using condition (\ref{prt:incycle}) it is easy to see that the extended $\varphi$ will be a partial automorphism.
		
		Now take the minimal $i$ with  $v_i \not \in \ran(\varphi)$ and follow the procedure outlined above, etc. Clearly, this process yields an automorphism $\varphi$ that witnesses that $f$ and $f'$ are conjugate, hence $C_0$ is a conjugacy class.
		
		Now we prove the second part of the proposition. By conditions (\ref{prt:nofini}), (\ref{prt:pattern}) and Theorem \ref{t:randomintro} the class $C_0$ is compact catcher. Moreover, observe that conditions (\ref{prt:nofini})-(\ref{prt:outcycle}) are invariant under taking inverses, hence $C_0=C^{-1}_0$. Now let $h \in \aut(\mc{R})$ be arbitrary. Then, using the fact that $C_0$ is compact catcher for the compact set $\{\id_{\mc{R}}, h\}$ there exists a $g \in \aut(\mc{R})$ such that $g,gh \in C_0$, in other words, $h \in C^{-1}_0C_0=C^2_0$, so $C^2_0=\aut(\mc{R})$ holds.

	\end{proof}
	
	In order to show Proposition \ref{p:c0prop1} one can use the methods from Section \ref{s:autr}. By the conjugacy invariance of $C^2$ it is enough to prove that for every $h \in C$ there exists a $g$ such that $g^{-1}hgh \in C_0$. Define $g$ inductively, maintaining the following properties: whenever we extend $g$ (or $g^{-1}$) to some $v$ vertex the vertex $g(v)$ should be far enough in $d_{\{h\}}$ from every vertex already used in the induction, be a splitting point for the compact set $\{h,\id_{\mc{R}}\}$ (note that this is equivalent to saying that $v$ is not a fixed point of $h$) and not connected to every already used vertex - except for those to which it is necessary in order for $g$ to be an automorphism.

	\bigskip
	\textbf{Acknowledgements.} We would like to thank to R. Balka, Z. Gyenis, A. Kechris, C. Rosendal, S. Solecki and P. Wesolek for many valuable remarks and discussions.
	
	\bibliographystyle{apalike}
	\bibliography{ran}

\end{document}